\documentclass[a4paper,12pt]{article}
\title{{\it A Priori} and {\it A Posteriori} Error Control of Discontinuous Galerkin Finite Element Methods for the Von K\'{a}rm\'{a}n Equations}
\author{Carsten Carstensen\footnote{
  Department of Mathematics, Humboldt-Universit\"{a}t zu Berlin, 10099 Berlin, Germany.
  Distinguished Visiting Professor, Department of Mathematics, Indian institute of Technology Bombay, Powai, Mumbai-400076. Email. cc@math.hu-berlin.de}        
$\;\;$ Gouranga Mallik\footnote{Department of Mathematics, Indian Institute of Technology Bombay, Powai, Mumbai 400076, India. Email. gouranga@math.iitb.ac.in}
$\;\;$ Neela Nataraj\footnote{Department of Mathematics, Indian Institute of Technology Bombay, Powai, Mumbai 400076, India. Email. neela@math.iitb.ac.in}
}
\usepackage{amsmath,amsthm,amssymb,enumerate}
\usepackage{gensymb}
\usepackage[square,sort&compress,comma,numbers]{natbib}

\usepackage[sc]{mathpazo}
\usepackage{multirow} 

\usepackage{newtxtext,newtxmath}
\usepackage{subfig}
\usepackage{graphicx}
\usepackage{epstopdf}
\usepackage{hyperref}
\usepackage{cancel}
\usepackage[margin=3cm]{geometry}
\usepackage{tikz}
\usepackage{caption}
\usetikzlibrary{shapes,calc}
\usepackage{verbatim}
\usepackage{mathrsfs}
\usepackage{accents}
\newcommand{\ubar}[1]{\underaccent{\bar}{#1}}

\chardef\bslash=`\\ 

\newtheorem{thm}{Theorem}[section]

\newtheorem{lem}[thm]{Lemma}

\theoremstyle{definition}
\newtheorem{defn}{Definition}[section]

\theoremstyle{remark}
\newtheorem{rem}{Remark}[section]

\numberwithin{equation}{section}

\newcommand{\bB}{\mathbb B}

\newcommand{\bN}{\mathbb N}
\newcommand{\bR}{\mathbb R}

\newcommand{\cA}{\mathcal A}

\newcommand{\cE}{\mathcal E}

\newcommand{\cT}{\mathcal T}

\newcommand{\cN}{\mathcal N}

\newcommand{\bH}{\boldsymbol{H}}

\newcommand{\bX}{\boldsymbol{X}}
\newcommand{\bS}{\boldsymbol{S}}

\newcommand{\lt}{L^2(\Omega)}

\newcommand{\hto}{H^2_0(\Omega)}

\newcommand{\bYh}{\boldsymbol{Y}_h}

\newcommand{\St}{P_2(\cT)}
\newcommand{\Stb}{{\boldsymbol{P}_2}(\cT)}

\newcommand{\vket}{von K\'{a}rm\'{a}n equations }

\newcommand{\integ}{\int_\Omega}

\newcommand{\sik}{\sum_{K\in\mathcal{T}}\int_K}

\newcommand{\divc}{\mathrm{div}}
\newcommand{\se}{\sum_{E\in \mathcal{E}}\int_E}
\newcommand{\sie}{\sum_{E\in \mathcal{E}(\Omega)}\int_E}

\newcommand{\ip}{{\rm IP}}

\newcommand{\fl}{\;\text{ for all }}

\newcommand{\half}{\frac{1}{2}}
\newcommand{\trinl}{\ensuremath{|\!|\!|}}
\newcommand{\trinr}{\ensuremath{|\!|\!|}}
\newcommand{\trinlNH}{\ensuremath{|\!|\!|}}
\newcommand{\trinrNH}{\ensuremath{|\!|\!|}}

\newcommand{\dx}{{\rm\,dx}}

\newcommand{\ds}{{\rm\,ds}}

\newcommand{\cof}{{\rm cof}}
\newcommand{\dg}{{\rm dG}}

\newcommand{\BFS}{Bogner-Fox-Schmit }

\newcommand{\Holder}{H\"{o}lder }

\begin{document}
\date{}	
\maketitle

\begin{abstract} 
This paper analyses discontinuous Galerkin finite element methods (DGFEM) to approximate a regular solution to the von K\'{a}rm\'{a}n equations defined on a polygonal
domain. A discrete inf-sup condition sufficient for the stability of the discontinuous
Galerkin discretization of a well-posed linear problem is established and this allows
the proof of local existence and uniqueness of a discrete solution to the non-linear problem
with a Banach fixed point theorem. The Newton scheme is locally second-order convergent and appears to be  a robust solution strategy up to machine precision. A comprehensive a~priori and a~posteriori energy-norm error analysis relies on one sufficiently large stabilization parameter and sufficiently fine triangulations. In case the other stabilization parameter
degenerates towards infinity, the DGFEM reduces to a novel $C^0$ interior penalty method (IPDG). In contrast to the known $C^0$-IPDG due to Brenner {\it et al} \cite{BS_C0IP_VKE}, the overall
discrete formulation maintains symmetry of the trilinear form in the first two components -- despite the general non-symmetry of the discrete nonlinear problems. Moreover, a reliable and efficient a~posteriori error analysis immediately follows for the DGFEM of this paper, while the
different norms in the known $C^0$-IPDG lead to complications with some non-residual type remaining terms. Numerical experiments confirm the best-approximation results and the
equivalence of the error and the error estimators. A related adaptive mesh-refining algorithm leads to optimal empirical convergence rates for a non convex domain.
\end{abstract}

\section{Introduction}
The discontinuous Galerkin finite element methods (DGFEM) have become popular for the numerical solution of a large range of problems in partial differential equations, which include linear and nonlinear problems, convected-dominated diffusion for second- and fourth-order elliptic problems.  Their advantages are well-known;  the flexibility offered by the discontinuous basis functions eases the global finite element assembly and the hanging nodes in mesh generation helps to handle complicated geometry. The continuity restriction for conforming FEM is relaxed, thereby making it an interesting choice for adaptive mesh refinements. On the other hand, conforming finite element methods for plate problems demand $C^1$ continuity and involve complicated higher-order finite elements. The simplest examples are Argyris finite element with $21$ degrees of freedom in a triangle and \BFS element with $16$ degrees of freedom in a rectangle. 

\medskip
Nonconforming \cite{Morley68}, mixed and hybrid \cite{BrezziFortin91,BoffiBrezziFortin13} finite element methods are also alternative approaches that have been used to relax the $C^1$-continuity. Discontinuous Galerkin methods are well studied for linear fourth-order elliptic problems, e.g., the $hp$-version of the nonsymmetric interior penalty DGFEM (NIPG) \cite{MozoSuli03}, the $hp$-version of the symmetric interior penalty DGFEM (SIPG) \cite{MozoSuliBosing07} and a combined analysis of NIPG and SIPG in \cite{MozoSuli07}. The literature on {\it a posteriori} error analysis for biharmonic problems with DGFEM include \cite{Georgoulis2011} and a quadratic $C^0$-interior penalty method \cite{BGS10}. The medius analysis in \cite{Gudi10} combines ideas of {\it a priori} and {\it a posteriori} analysis to establish error estimates for DGFEM under minimal regularity assumptions on the exact solution.

\medskip
This paper concerns discontinuous Galerkin finite element methods for the approximation of regular solution to the von K\'{a}rm\'{a}n equations, which describe the deflection of very thin elastic plates. Those plates are  modeled by a semi-linear system of fourth-order partial differential equations and can be described as follows. For a given load function $f\in\lt$, seek $u,v \in H^2_0(\Omega)$ such that
\begin{equation}\label{vkedG}
\Delta^2 u =[u,v]+f \text{ and }\Delta^2 v =-\half[u,u] \quad
 \text{ in } \Omega
\end{equation}
with the biharmonic operator $\Delta^2$ and the von K\'{a}rm\'{a}n bracket $[\bullet,\bullet]$,
$\displaystyle\Delta^2\varphi:=\varphi_{xxxx}+2\varphi_{xxyy}+\varphi_{yyyy}$, and $\displaystyle
[\eta,\chi]:=\eta_{xx}\chi_{yy}+\eta_{yy}\chi_{xx}-2\eta_{xy}\chi_{xy}=\cof(D^2\eta):D^2\chi$ for the co-factor matrix $\cof(D^2\eta)$ of $D^2\eta$. The colon $:$ denotes the scalar product of two $2\times 2$ matrices.

\medskip \noindent
In \cite{Brezzi}, conforming finite element approximations for the von K\'{a}rm\'{a}n equations are analyzed and an error estimate in energy norm is derived for approximations of regular solutions. Mixed and hybrid methods reduce the system of fourth-order equations into a system of second-order equations \cite{Miyoshi,BrezziRappazRaviart80,BrezziRappazRaviart81,Reinhart}. Conforming finite element methods for the canonical von K\'{a}rm\'{a}n equations have been proposed and error estimates in energy, $H^1$ and $L^2$ norms are established in \cite{GMNN_BFS} under realistic regularity assumption on the exact solution. Nonconforming FEMs have also been analyzed for this problem \cite{GMNN_Morley}. An {\it a priori} error analysis for a $C^0$ interior penalty method of this problem is studied in \cite{BS_C0IP_VKE}. Recently, an abstract framework for nonconforming discretization of a class of semilinear elliptic problems which include \vket is analyzed in \cite{CCGMNN_Semilinear}. 

\medskip
In this paper, discontinuous Galerkin finite element methods are applied to approximate the regular solutions of the von K\'{a}rm\'{a}n equations. To highlight the contribution, under  minimal regularity assumption of the exact solution, optimal order {\it a priori} error estimates are obtained and a reliable and efficient a posteriori error estimator is designed. Moreover, {\it a priori} and {\it a posteriori} error estimates for a $C^0$ interior penalty method for the von K\'{a}rm\'{a}n equations are recovered as a special case. The comprehensive {\it a priori} analysis in \cite{BS_C0IP_VKE} controls the error in the stronger norm $\|\cdot\|_h\equiv \|\bullet\|_{\widetilde{\ip}}$ and therefore requires a more involved mathematics and a trilinear form $b_{\widetilde{IP}}$ without symmetry in the first two variables, cf. Remark~\ref{Brenner_b}.

\medskip \noindent 
The remaining parts of the paper are organized as follows. Section 2 describes some preliminary results and introduces  discontinuous Galerkin finite element methods for von K\'{a}rm\'{a}n equations. Section 3 discusses some auxiliary results required for {\it a priori} and {\it a posteriori} error analysis. In Section~4, a discrete inf-sup condition is established for a linearized problem for the proof of the existence, local uniqueness and error estimates of the discrete solution of the non-linear problem. In Section~5, a reliable and efficient {\it a posteriori} error estimator is derived. Section~6 derives {\it a priori} and {\it a posteriori} error estimates for a $C^0$ interior penalty method.  Section~7 confirms the theoretical results in various numerical experiments and establishes an adaptive mesh-refining algorithm. 

\medskip
Throughout the paper, standard notation on Lebesgue and Sobolev spaces and their norms are employed.
The standard semi-norm and norm on $H^{s}(\Omega)$ (resp. $W^{s,p} (\Omega)$) for $s>0$ are denoted by $|\bullet|_{s}$ and $\|\bullet\|_{s}$ (resp. $|\bullet|_{s,p}$ and $\|\bullet\|_{s,p}$ ).  Bold letters refer to vector valued functions and spaces, e.g. $\bX= X\times X$. The positive constants $C$ appearing in the inequalities denote generic constants which do not depend on the mesh-size. The notation $A\lesssim B$ means that there exists a generic constant $C$ independent of the mesh parameters and independent of the stabilization parameters $\sigma_1$ and $\sigma_2\geq 1$ such that $A \leq CB$; $A\approx B$ abbreviates $A\lesssim B\lesssim A$.

\section{Preliminaries}
This section introduces weak and discontinuous Galerkin (dG) formulations for the von K\'{a}rm\'{a}n equations.

\subsection{Weak formulation}
The weak formulation of von K\'{a}rm\'{a}n equations \eqref{vkedG} reads: Given $f\in\lt$, seek  $u,v\in \: X:=\hto$ such that
\begin{subequations}\label{wform}
	\begin{align}
	& a(u,\varphi_1)+ b(u,v,\varphi_1)+b(v,u,\varphi_1)=l(\varphi_1)   \fl\varphi_1\in X\label{wforma}\\
	& a(v,\varphi_2)-b(u,u,\varphi_2)   =0            \fl\varphi_2 \in X.\label{wformb}
	\end{align}
\end{subequations}
Here and throughout the paper, for all $\eta,\chi,\varphi\in X$, 
\begin{align}
&a(\eta,\chi):=\integ D^2 \eta:D^2\chi\dx,\; \; b(\eta,\chi,\varphi):=-\half\integ [\eta,\chi]\varphi\dx, \text{ and } l(\varphi):=\int_{\Omega}f\varphi\dx. \label{defnab}
\end{align}
Given $F=(f,0)\in L^2(\Omega)\times L^2(\Omega)$, the combined vector form seeks $\Psi=(u,v)\in \bX:=X\times X\equiv\hto\times\hto$ such that
\begin{equation}\label{vform_cts_dG}
N(\Psi;\Phi):=A(\Psi,\Phi)+B(\Psi,\Psi,\Phi)-L(\Phi)=0\fl \Phi\in \bX,
\end{equation}
where, for all $\Xi=(\xi_1,\xi_2),\Theta=(\theta_1,\theta_2)$, and $\Phi=(\varphi_1,\varphi_2)\in  \bX$,
\begin{align*}
& A(\Theta,\Phi):=a(\theta_1,\varphi_1)+a(\theta_2,\varphi_2),\\
&B(\Xi,\Theta,\Phi):=b(\xi_1,\theta_2,\varphi_1)+b(\xi_2,\theta_1,\varphi_1)-b(\xi_1,\theta_1,\varphi_2)\text{ and } L(\Phi):=l(\varphi_1).
\end{align*}
Let $\trinl\bullet\trinr_2$ denote the product norm on $\bX$  defined by $\trinl\Phi\trinr_2:=\left(|\varphi_1|_{2,\Omega}^2+|\varphi_2|_{2,\Omega}^2\right)^{1/2}$ for all $\Phi=(\varphi_1,\varphi_2)\in \bX$. It is easy to verify that the following boundedness and ellipticity properties hold
\begin{align*}
&{A}(\Theta,\Phi)\leq \trinl\Theta\trinr_2 \: \trinl\Phi\trinr_2,\: {A}(\Theta,\Theta) \geq \trinl\Theta\trinr_2^2,\\
&\quad B(\Xi, \Theta, \Phi) \leq  C \trinl\Xi\trinr_2 \: \trinl\Theta\trinr_2 \: \trinl\Phi\trinr_2.
\end{align*}
Since $b(\bullet,\bullet,\bullet)$ is symmetric in first two variables, the trilinear form $B(\bullet,\bullet,\bullet)$ is symmetric in first two variables.

\medskip
For results regarding the existence of solution to \eqref{vform_cts_dG}, regularity and bifurcation phenomena, we refer  to \cite{CiarletPlates, Knightly, BergerFife66, Berger,BergerFife, BlumRannacher}. It is well known \cite{BlumRannacher} that on a polygonal domain $\Omega$, for given $f\in H^{-1}(\Omega)$, the solutions $u,v$ belong to $\hto\cap H^{2+\alpha}(\Omega)$, for the index of elliptic regularity $\alpha\in (\half,1]$ determined by the interior angles of $\Omega$. Note that when $\Omega$ is convex; $\alpha=1$; that is, the solution belongs to $\hto\cap H^3(\Omega) $. Unless specified otherwise, the parameter $\alpha$ is supposed to satisfy $1/2<\alpha\leq 1$. 

\medskip
Throughout the paper, we consider the approximation of a regular solution \cite{Brezzi,GMNN_BFS} $\Psi$ to the non-linear operator $N(\Psi; \Phi)=0 \; $ for all $\Phi \in {\bX}$  of \eqref{vform_cts_dG} in the sense that the bounded derivative $DN(\Psi)$ of the operator $N$ at the solution $\Psi$ is an isomorphism in the Banach space; this is equivalent to an inf-sup condition 
\begin{align}\label{inf-sup_dG}
0<\beta:=\inf_{\substack{\Theta\in \bX\\ \trinl\Theta\trinr_2=1}}\sup_{\substack{\Phi\in \bX\\ \trinl\Phi\trinr_2=1}}\big{(}A(\Theta,\Phi)+2B(\Psi,\Theta,\Phi)\big{)}.
\end{align}

\subsection{Triangulations}
Let  $\cT$ be a shape-regular \cite{Braess} triangulation of $\Omega$ into closed triangles. 
The set of all internal vertices (resp. boundary vertices) and  interior edges (resp.  boundary edges)  of the triangulation $\cT$ are denoted by $\cN (\Omega)$ (resp.  $\cN(\partial\Omega)$) and $\cE (\Omega)$ (resp. $\cE (\partial\Omega)$).
Define a piecewise constant mesh function $h_{\cT}(x)=h_K={\rm diam} (K)$ for all $x \in K$, $ K\in \cT$, and set $h:=\max_{K\in \cT}h_K$. Also define a piecewise constant edge-function on $\cE:=\cE(\Omega)\cup \cE(\partial\Omega)$ by $h_{\cE}|_E=h_E={\rm diam}(E)$ for any $E\in \cE$. Set of all edges of $K$ is denoted by $\cE(K)$. Note that for a shape-regular family, there exists a positive constant $C$ independent of $h$ such that any $K\in\cT$ and any $E\in \partial K$ satisfy 
\begin{equation}\label{shape_reg_const}
Ch_K\leq h_E\leq h_K.
\end{equation}
\noindent Let $P_r( K)$ denote the set of all polynomials of degree less than or equal to $r$ and $\displaystyle P_r(\cT):=\left\{\varphi\in L^2(\Omega):\,\forall K\in\cT,\varphi|_{K}\in P_r(K)\right\}$ and write $\boldsymbol{P}_r(\cT):=P_r(\cT)\times P_r(\cT)$ for pairs of piecewise polynomials.
For a nonnegative integer $s$, define the broken Sobolev space for the subdivision $\cT$  as
\begin{equation*}
H^s(\cT)=\left\{\varphi\in\lt: \varphi|_K\in H^{s}(K)\fl K\in \cT \right\}
\end{equation*}
with the broken Sobolev semi-norm  $|\bullet|_{H^s(\cT)}$ and norm $\| \bullet\|_{H^s(\cT)}$ defined by
\begin{equation*}
|\varphi|_{H^s(\cT)}=\bigg{(}\sum_{K\in\cT} |\varphi|_{H^{s}(K)}^2\bigg{)}^{1/2}\text{ and }
\|\varphi\|_{H^s(\cT)}=\bigg{(}\sum_{K\in\cT}\|\varphi\|_{H^{s}(K)}^2\bigg{)}^{1/2}.
\end{equation*}
Define the jump $[\varphi]_E=\varphi|_{K_+}-\varphi|_{K_-}$ and the average $\langle\varphi\rangle_E=\half\left(\varphi|_{K_+}+\varphi|_{K_-}\right)$ across the interior edge $E$ of $\varphi\in H^1(\cT)$ of the adjacent triangles  $K_+$ and $K_-$. Extend the definition of the jump and the average to an edge lying in boundary by $[\varphi]_E=\varphi|_E$ and $\langle\varphi\rangle_E=\varphi|_E$ for $E\in \cE(\partial\Omega)$ owing to the homogeneous boundary conditions.  For any vector function, jump and average are understood componentwise.
Set $\Gamma\equiv\bigcup_{E\in\cE}E$. 

\subsection{Discrete norms and bilinear forms}
\noindent For $1/2<\alpha\leq 1$, abbreviate $Y_h:=(X\cap H^{2+\alpha}(\Omega))+\St$ and $\bYh:=Y_h\times Y_h$. For all $\eta,\chi\in Y_h$, $\varphi\in X+\St$, introduce the bilinear, trilinear and linear forms by
\begin{align*}
& a_{\dg}(\eta,\chi):=\sum_{K\in\cT}\int_K D^2\eta:D^2\chi\dx-\left(J(\eta,\chi)+ J(\chi,\eta)\right) +J_{\sigma_1,\sigma_2}(\eta,\chi),\\
& b_{\dg}(\eta,\chi,\varphi):=-\half \sum_{K\in\cT}\int_K [\eta,\chi]\varphi
\dx, \quad  l_{\dg}(\varphi):= \sum_{K\in\cT}\int_K  f\varphi\dx,\\
&J(\eta,\chi)=\sum_{E\in\cE}\int_E [\nabla\chi]_E\cdot\langle D^2\eta\;\nu_E\rangle_E \ds,
\end{align*}
with $\sigma_1>0$ and $\sigma_2>0$ to be suitably chosen in the jump terms across any edge $E\in\cE$ with unit normal vector $\nu_E$ and
\begin{align*}
& J_{\sigma_1,\sigma_2}(\eta,\chi):=\sum_{E\in\cE}\frac{\sigma_1}{h_E^3}\int_E[\eta]_E[\chi]_E\ds+\sum_{E\in\cE}\frac{\sigma_2}{h_E}\int_E[\nabla \eta\cdot\nu_E]_E[\nabla \chi\cdot\nu_E]_E\ds.
\end{align*}

\medskip

The discontinuous Galerkin (dG) finite element formulation of \eqref{vkedG} seeks $(u_{\dg},v_{\dg})\in \Stb:=\St\times \St$ such that, for all $(\varphi_1, \varphi_2)\in \Stb$,
\begin{align}
&a_{\dg}(u_{\dg},\varphi_1)+b_{\dg}(u_{\dg},v_{\dg},\varphi_1)+b_{\dg}(v_{\dg},u_{\dg},\varphi_1)=l_{\dg}(\varphi_1), \label{wformdg1}\\
& a_{\dg}(v_{\dg},\varphi_2)-b_{\dg}(u_{\dg},u_{\dg},\varphi_2)=0. \label{wformdg2}
\end{align}
\noindent The combined vector form seeks $\Psi_{\dg}\equiv (u_{\dg},v_{\dg})\in \Stb$ such that, for all $\Phi_{\dg}\in \Stb$,
\begin{equation}\label{vform_d_dG}
N_h(\Psi_{\dg};\Phi_{\dg}):=A_{\dg}(\Psi_{\dg},\Phi_{\dg})+B_{\dg}(\Psi_{\dg},\Psi_{\dg},\Phi_{\dg})-L_{\dg}(\Phi_{\dg})=0,
\end{equation}
where, $\fl\Xi_{\dg}=(\xi_1, \xi_2),\Theta_{\dg}=(\theta_1,\theta_2),\Phi_{\dg}=(\varphi_1,\varphi_2)\in \Stb$,
\begin{align}
&A_{\dg}(\Theta_{\dg},\Phi_{\dg}):=a_{\dg}(\theta_1,\varphi_1)+a_{\dg}(\theta_2,\varphi_2),\\ 
&B_{\dg}(\Xi_{\dg},\Theta_{\dg},\Phi_{\dg}):=b_{\dg}(\xi_{1},\theta_{2},\varphi_{1})+b_{\dg}(\xi_{2},\theta_{1},\varphi_{1})-b_{\dg}(\xi_{1},\theta_{1},\varphi_{2}), \\
&L_{\dg}(\Phi_{\dg}):=l_{\dg}(\varphi_1).
\end{align}
Note that $b_{\dg}(\bullet,\bullet,\bullet)$ is symmetric in the first and second variables, and so is $B_{\dg}(\bullet,\bullet,\bullet)$.

\medskip
\noindent For $\varphi\in H^2(\cT)$ and $\Phi=(\varphi_1,\varphi_2)\in \bH^2(\cT)\equiv H^2(\cT)\times H^2(\cT)$, define the mesh dependent  norms $\|\bullet\|_{\dg}$ and $\trinl\bullet\trinr_{\dg}$ by 
\begin{align*}
&\|\varphi\|_{\dg}^2:=|\varphi|_{ H^2(\cT)}^2+\sum_{E\in\cE}\frac{\sigma_1}{h_E^3}\|[\varphi]_E\|_{L^2(E)}^2+\sum_{E\in\cE}\frac{\sigma_2}{h_E}\|[\nabla \varphi\cdot\nu_E]_E\|_{L^2(E)}^2,\\
&\trinl\Phi\trinr_{\dg}^2:=\|\varphi_1\|_{\dg}^2+\|\varphi_2\|_{\dg}^2.
\end{align*} 

\noindent For $\xi\in Y_h\equiv (X\cap H^{2+\alpha}(\Omega))+\St$ and $\Xi=(\xi_1,\xi_2)\in \bYh\equiv Y_h\times Y_h$, define the auxiliary norms $ \|\bullet\|_h$ and $\trinl\bullet\trinr_h$  by
\begin{align*}
\|\xi\|_h^2:=\|\xi\|_{\dg}^2+\sum_{E\in\cE}\sum_{j,k=1}^{2}\|h_E^{1/2}\langle \partial^2\xi/\partial x_j \partial x_k\rangle_E  \|_{L^2(E)}^2\text{ and } \trinl\Xi\trinr_h^2:=\|\xi_1\|_h^2+\|\xi_2\|_h^2.
\end{align*}

\section{Auxiliary results }
This section discusses some auxiliary results and establishes the boundedness and ellipticity results required for the analysis.
\subsection{Some known operator bounds}
\noindent This subsection recalls a few standard results. Throughout this subsection, the generic multiplicative constant $C\approx 1$ hidden in the brief notation $\lesssim$ depends on the shape regularity of the triangulation $\cT$ and arising parameters like the polynomial degree $r\in \bN_0$ or the Lebesgue index $p$ and the Sobolev indices $\ell,s>1/2$ and $1/2<\alpha\le 1$; $C$ is independent of the mesh-size.

\begin{lem}[Inverse inequality I]\label{inv_ineq_ncfem}\label{inv_ineq_dg} \cite{LasisSuli03,Brenner} There exists a positive constant $C$ that depends on $\ell\in\mathbb{N},\,p\in \bR$ and the shape regularity of the triangulation such that, for $1\leq \ell,\, 2\leq p\leq \infty$, for any $\xi\in P_r(K)$ satisfies
	\begin{align*}
	\|\xi\|_{L^p(K)}&\lesssim h_K^{(2-p)/p}\|\xi\|_{L^2(K)}\text{ and }|\xi|_{H^{\ell}(K)}\lesssim h_K^{-1}|\xi|_{H^{{\ell}-1}(K)},
	\end{align*}	
	for any $K \in \cT$ with $E\subset \cE(K)$, where
	\begin{equation*}
	\|\xi\|_{L^p(E)}\lesssim h_E^{1/p-1/2}\|\xi\|_{L^2(E)}.
	\end{equation*}
\end{lem}
\begin{lem}[Trace inequality]\label{trace_inq} The following trace inequalities hold for $K\in\cT$ and $s>1/2$.
	\begin{enumerate}[(a)]
		\item  \cite{DiPetroErn12} $\;\displaystyle
		\|\xi\|_{L^2(\partial K)}\lesssim h_K^{-1/2}\|\xi\|_{L^2(K)}\fl \xi\in P_r(K)$;
		\item  \cite[p. 111]{BrennerOwensSung08}
       $\;\displaystyle \|\xi\|_{L^2(\partial K)}\lesssim {h_K^{s-1/2}}\|\xi\|_{H^s(K)}+h_K^{-1/2}\|\xi\|_{L^2(K)}\text{ for all }\xi\in H^s(K)$.	
	\end{enumerate}	
\end{lem}

\begin{lem}[Interpolation estimates]\label{interpolant_dG} \cite{BabuskaSuri87} There exists a linear operator
$\Pi_h: H^s(\cT)\to P_r(\cT)$, such that, for $0\leq q\leq s,\;m=\min(r+1,s)$ and $1/2<\alpha\leq 1$,
\begin{align}
\|\varphi-\Pi_{h}\varphi\|_{H^{q}(K)}&\lesssim h_K^{m-q}\|\varphi\|_{H^{s}(K)}\fl K\in\cT\text{ and}\fl\varphi\in H^s(\cT),\label{local_int_est}\\
\|\varphi-\Pi_h\varphi\|_{\dg}&\leq\|\varphi-\Pi_h\varphi\|_{h}\lesssim h^{\alpha}\|\varphi\|_{2+\alpha}\fl\varphi\in H^{2+\alpha}(\Omega).\label{interpolation_est}
\end{align}
\end{lem}
\begin{proof}
The proof of \eqref{interpolation_est} follows from Lemma~\ref{trace_inq}.b, \eqref{local_int_est}, and an interpolation of Sobolev spaces \cite[Subsection 14.1]{Brenner}.
\end{proof}
\noindent For an easy notation of vectors, we denote the componentwise interpolation of $\boldsymbol{\zeta}\in\boldsymbol{H}^s(\cT):=H^s(\cT)\times H^s(\cT)$ by $\Pi_h\boldsymbol{\zeta}$.

\begin{figure}
	\begin{center}
		\begin{tikzpicture}[scale=2]
		\draw[line width=0.25mm] (0,0)--(2,0)--(1,1.414)--(0,0);
		\foreach \Point in {(0,0), (2,0), (1,1.414), (1,0),(1.5,0.7),(0.5,.7)}{\node at \Point {$\bullet$};}
		\foreach \Point in {(3,0), (5,0), (4,1.414),  (4,0),(4.5,.7), (3.5,.7),(4,0.472)}{\node at \Point {$\bullet$};}
		\node at (1,0.5) {$K$};
		
		\draw[line width=0.25mm] (3,0)--(5,0)--(4,1.414)--(3,0);
		\draw(3,0) circle(1mm);\draw(5,0) circle(1mm);\draw(4,1.414) circle(1mm);\draw(4,0.477) circle(1mm);
		\draw[line width=0.25mm] (3,0)--(4,0.48)--(5,0); \draw[line width=0.25mm] (4,0.48)--(4,1.414);
		\draw[->,line width=0.25mm] (4.25,1.05)--(4.5,1.25);\draw[->,line width=0.25mm] (4.75,.35)--(5,.5);
		\draw[->,line width=0.25mm] (3.75,1.05)--(3.5,1.25);\draw[->,line width=0.25mm] (3.25,.35)--(3,.5);
		\draw[->,line width=0.25mm] (3.5,0)--(3.5,-.3);\draw[->,line width=0.25mm] (4.5,0)--(4.5,-.3);
		
		\node at (4,0.2) {$K_1$};
		\node at (4.3,.6) {$K_2$};
		\node at (3.7,0.57) {$K_3$};
		\end{tikzpicture}
		\caption{A $P_2$ Lagrange triangular element and a $\tilde{P}_4$- $C^1$- conforming macro element}
		\label{fig:macro_elem}
	\end{center}
\end{figure}
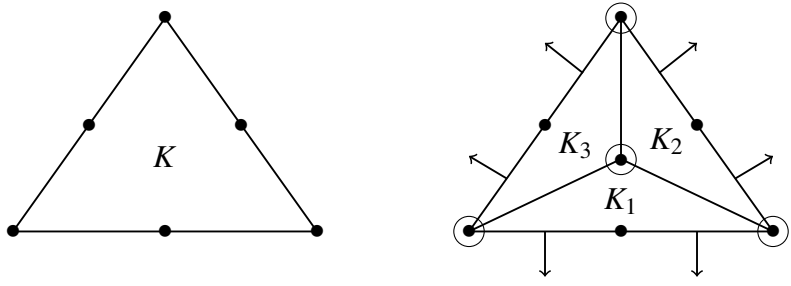

\begin{defn}
	\cite{Georgoulis2011} For $K\in\cT$, a macro-element of degree $4$ is a nodal finite element $(K,\tilde{P}_4,\tilde{N})$, consisting of sub-triangles $K_j$, $j=1,2,3$ (see Figure~\ref{fig:macro_elem}). The local element space $\tilde{P}_4$ is defined by 
	\begin{align*}
	\tilde{P}_4:=\left\{\varphi\in C^1(K): \varphi|_{K_j}\in P_4(K_j),\;j=1,2,3\right\}.
	\end{align*}	
	The degrees of freedom $\tilde{N}$ are defined as (a) the value and the first (partial) derivatives at the vertices of $K$; (b) the value at the midpoint of each edge of $K$; (c) the normal derivative at two distinct points in the interior of each edge of $K$; (d) the value and the first (partial) derivatives at the common vertex of  $K_1, K_2$ and $K_3$. The corresponding finite element space consisting of the above macro-elements will be denoted by $S_4(\cT)\subset \hto$.

\end{defn}

The enrichment operator of \cite{Georgoulis2011} is outlined in the sequel for a convenient reading. For each nodal point $p$ of the $C^1$-conforming finite element space $S_4(\cT)$, define $\cT(p)$ to be the set of $K\in\cT$ which shares the nodal point $p$ and let $|\cT(p)|$ denotes its cardinality. Define the operator $E_h:\St\to S_4(\cT)$ for any nodal variable $N_p$ at $p$ by
\begin{align*}
N_p(E_h(\varphi_{\dg})):=\begin{cases}
\frac{1}{|\cT(p)|}\sum_{K\in\cT(p)} N_p(\varphi_{\dg}|_K)\;\text{ if } p\in  \cN(\Omega),\\
0\; \text{ if } p\in \cN(\partial\Omega).
\end{cases}
\end{align*}

\begin{lem}[Enrichment operator]\cite{Georgoulis2011}\label{dG_Enrich_apost}\label{enrichment_dG} The enrichment operator $E_h: \St\to S_4(\cT)$ satisfies, for $m=0,1,2$,
\begin{align}\label{enrich_apost}
\sum_{K\in\cT}\left|\varphi_{\dg}-E_h\varphi_{\dg}\right|_{H^m(K)}^2 & \lesssim\|h_{\cE}^{1/2-m}[\varphi_{\dg}]_{\cE}\|_{L^2(\Gamma)}^2+\|h_{\cE}^{3/2-m}[\nabla \varphi_{\dg}]_{\cE}\|_{L^2(\Gamma)}^2 \lesssim h^{2-m}\|\varphi_{\dg}\|_{\dg}.
\end{align}
		Moreover, for some positive constant $\Lambda\approx 1$,
		\begin{align}\label{EhLambda}
		\|\varphi_{\dg}-E_h\varphi_{\dg}\|_{\dg}\leq \Lambda \inf_{\varphi\in X}\|\varphi_{\dg}-\varphi\|_{\dg}.
		\end{align}
	\end{lem}
	\begin{proof}
		See \cite[Lemma 3.1]{Georgoulis2011} for a proof of \eqref{enrich_apost}. For the proof of \eqref{EhLambda}, choose $m\leq 2$ in \eqref{enrich_apost} and obtain (with $h_{\cE}\lesssim h\lesssim 1$) that
		\begin{align*}
		\|\varphi_{\dg}-E_h\varphi_{\dg}\|_{ H^2(\cT)}^2\lesssim\|h_{\cE}^{-3/2}[\varphi_{\dg}]_{\cE}\|_{L^2(\Gamma)}^2+\|h_{\cE}^{-1/2}[\nabla \varphi_{\dg}]_{\cE}\|_{L^2(\Gamma)}^2.
		\end{align*}
		Since $\left[\varphi_{\dg}-E_h\varphi_{\dg}\right]_E=[\varphi_{\dg}]_E$ and $\left[\nabla(\varphi_{\dg}-E_h\varphi_{\dg})\right]_E=\left[\nabla \varphi_{\dg}\right]_E$,  those edge terms in both sides of the above inequality lead (in the definition of $\|\cdot\|_{\dg}$) to
		\begin{align*}
		\|\varphi_{\dg}-E_h\varphi_{\dg}\|_{\dg}^2\lesssim \|h_{\cE}^{-3/2}[\varphi_{\dg}]_{\cE}\|_{L^2(\Gamma)}^2+\|h_{\cE}^{-1/2}[\nabla \varphi_{\dg}]_{\cE}\|_{L^2(\Gamma)}^2.
		\end{align*}
		Furthermore, any $\varphi\in X$ satisfies (with \eqref{EhLambda} for $m=2$ in the end) that
		\begin{align*}
		\|\varphi_{\dg}-E_h\varphi_{\dg}\|_{\dg}^2\lesssim \|h_{\cE}^{-3/2}[\varphi_{\dg}-\varphi]_{\cE}\|_{L^2(\Gamma)}^2+\|h_{\cE}^{-1/2}[\nabla (\varphi_{\dg}-\varphi)]_{\cE}\|_{L^2(\Gamma)}^2\lesssim \|\varphi_{\dg}-\varphi\|_{\dg}.
		\end{align*}
		This completes the proof of \eqref{EhLambda} for some $h$-independent positive constant $\Lambda$.
	\end{proof}
	

\begin{lem}[Inverse inequalities II]\label{auxest}
		It holds
		\begin{align*}
		\|h_{\cT}\nabla\varphi\|_{L^{\infty}(\Omega)}&\lesssim\|\varphi\|_{L^{\infty}(\Omega)} \fl\varphi\in \St+S_4(\cT),\\
		\|\varphi\|_{W^{1,4}(\cT)}+\|\varphi\|_{L^{\infty}(\Omega)}&\lesssim\|\varphi\|_{\dg}\fl\varphi\in \St+X.
		\end{align*}
	\end{lem}
\begin{proof}
This follows with the arguments of \cite[Lemma 3.7]{BS_C0IP_VKE} on the enrichment and interpolation operator. Further details are omitted for brevity.
\end{proof}

\subsection{Continuity and ellipticity}
%
This subsection is devoted to the boundedness and ellipticity results for the bilinear form $a_{\dg}(\bullet, \bullet)$ and boundedness results for $b_{\dg}(\bullet, \bullet,\bullet)$.


\begin{lem}[Boundedness of $a_{\dg}(\bullet,\bullet)$] Any $\theta_{\dg},\varphi_{\dg}\in \St+S_4(\cT)$ satisfies
	\begin{equation*}
	a_{\dg}(\theta_{\dg},\varphi_{\dg})\lesssim \|\theta_{\dg}\|_{\dg}\|\varphi_{\dg}\|_{\dg}.
	\end{equation*}
	
\end{lem}

\begin{proof}
	Given any $\theta_{\dg},\varphi_{\dg} \in \St$, recall the definition of $a_{\dg}(\bullet,\bullet)$
	\begin{align*}
	a_{\dg}(\theta_{\dg},\varphi_{\dg})& =\sik D^2\theta_{\dg}: D^2\varphi_{\dg}\dx-\left(J(\theta_{\dg},\varphi_{\dg})+ J(\varphi_{\dg},\theta_{\dg})\right) \nonumber \\
	& \qquad +J_{\sigma_1,\sigma_2}(\theta_{\dg},\varphi_{\dg}).
	\end{align*}
	The definition of $J(\bullet, \bullet)$, the Cauchy-Schwarz inequality, and Lemma~\ref{trace_inq} imply
	\begin{align} 
	&J(\theta_{\dg},\varphi_{\dg})=\se [\nabla\varphi_{\dg}]_E\cdot\langle D^2\theta_{\dg}\nu_E\rangle_E \ds\notag \\
	&\leq\sigma_2^{-1/2}\Big{(}\sum_{E\in\cE}\frac{\sigma_2}{h_E}\|[\nabla\varphi_{\dg}]_E\|_{L^2(E)}^2\Big{)}^{1/2}\Big{(}\sum_{E\in\cE}\|h_E^{1/2}\langle D^2\theta_{\dg}\rangle_E  \|_{L^2(E)}^2\Big{)}^{1/2}\label{Norm_bdd}\\
	&\lesssim \sigma_2^{-1/2} \|\varphi_{\dg}\|_{\dg}|\theta_{\dg}|_{H^2(\cT)}\leq \sigma_2^{-1/2} \|\theta_{\dg}\|_{\dg}\|\varphi_{\dg}\|_{\dg}.\label{J_bdd}
	\end{align}
	\noindent The same arguments show, $J(\varphi_{\dg},\theta_{\dg})\lesssim \sigma_2^{-1/2}\|\theta_{\dg}\|_{\dg}\|\varphi_{\dg}\|_{\dg}$. The definitions of $J_{\sigma_1,\sigma_2}(\bullet,\bullet)$ and $\|\bullet\|_{\dg}$, and the Cauchy-Schwarz inequality lead to 
	\begin{align*}
	&J_{\sigma_1,\sigma_2}(\theta_{\dg},\varphi_{\dg})=\sum_{E\in\cE}\frac{\sigma_1}{h_E^3}\int_E[\theta_{\dg}]_E[\varphi_{\dg}]_E\ds+\sum_{E\in\cE}\frac{\sigma_2}{h_E}\int_E[\nabla \theta_{\dg}\cdot\nu_E]_E[\nabla \varphi_{\dg}\cdot\nu_E]_E\ds\\
	&\leq \left(\sum_{E\in\cE}\frac{\sigma_1}{h_E^3}\|[\theta_{\dg}]_E\|_{L^2(E)}^2\right)^{1/2}\left(\sum_{E\in\cE}\frac{\sigma_1}{h_E^3}\|[\theta_{\dg}]_E\|_{L^2(E)}^2\right)^{1/2}\\
	&\quad+\left(\sum_{E\in\cE}\frac{\sigma_2}{h_E}\|[\nabla \theta_{\dg}\cdot\nu_E]_E\|_{L^2(E)}^2\right)^{1/2}\left(\sum_{E\in\cE}\frac{\sigma_2}{h_E}\|[\nabla \varphi_{\dg}\cdot\nu_E]_E\|_{L^2(E)}^2\right)^{1/2}\lesssim \|\theta\|_{\dg}\|\varphi\|_{\dg}.
	\end{align*}
	The combination of all displayed formulas and $\sigma_2\geq 1$ conclude the proof.
\end{proof}

\begin{rem}\label{a_bdd_h}
	The definitions of $a_{\dg}(\bullet,\bullet)$, the auxiliary norm $\|\bullet\|_h$, and the estimate \eqref{Norm_bdd} imply (since $\sigma_2\geq 1$)
	\begin{equation*}
	a_{\dg}(\theta,\varphi)\lesssim\|\theta\|_{h}\|\varphi\|_{h}\fl\theta,\varphi\in Y_h\equiv (X\cap H^{2+\alpha}(\Omega))+\St.
	\end{equation*}
\end{rem}

\begin{rem}	
	The trace inequality Lemma~\ref{trace_inq}.a implies that $\displaystyle\|\bullet\|_{\dg}\approx\|\bullet\|_{h}$ are equivalent norms on $\St+S_4(\cT)$ with equivalence constants, which do neither depend on the mesh-size nor on $\sigma
	_1,\sigma_2>0$.
\end{rem}
\begin{lem}[Ellipticity of $a_{\dg}(\bullet,\bullet)$]\label{ellipticity} 
	For any $\sigma_1>0$ and for a sufficiently large parameter $\sigma_2$, there exists some $h$-independent positive constant $\beta_0$ (which depends on $\sigma_2$) such that
	\begin{equation*}
	\beta_0 \|\theta_{\dg}\|_{\dg}^2 \leq a_{\dg}(\theta_{\dg},\theta_{\dg}) \fl\theta_{\dg}\in\St.
	\end{equation*}
\end{lem}
\begin{proof}
	For  $\theta_{\dg} \in\St$, the definition of $a_{\dg}(\bullet,\bullet)$ leads to
	\begin{align*}
	a_{\dg}(\theta_{\dg},\theta_{\dg})=\|\theta_{\dg}\|_{\dg}^2-2J(\theta_{\dg},\theta_{\dg}).
	\end{align*}
Recall \eqref{J_bdd} in the form 
$\displaystyle J(\theta_{\dg},\theta_{\dg})\leq C_0 \sigma_2^{-1/2}\|\theta_{\dg}\|_{\dg}^2$
with some constant $C_0\approx 1$. For any $0<\beta_0<1$ and any choice of $\sigma_2\geq 4C_0(1-\beta_0)^{-2}$, the combination of the previous estimates concludes the proof.
\end{proof}


%
%
Recall that $h$ denotes the maximal mesh-size of the underlying triangulation $\cT$.
\begin{lem}\label{intermediate}
	Any $\xi\in H^{2+\alpha}(\Omega)$ with $1/2<\alpha \leq 1$ and $\varphi_{\dg}\in\St$ satisfy
	\begin{equation}\label{Inpol_BTW}
	a_{\dg}(\xi,\varphi_{\dg}-E_h\varphi_{\dg})\lesssim h^\alpha \|\xi\|_{2+\alpha}\|\varphi_{\dg}\|_{\dg}.
	\end{equation}
	Consequently, for $\boldsymbol{\xi}\in  \bH^{2+\alpha}(\Omega)$ and $\Phi_{\dg} \in\Stb$,
	\begin{equation}\label{Inpol_BTW_A}
	A_{\dg}(\boldsymbol{\xi},\Phi_{\dg}-E_h\Phi_{\dg})\lesssim h^\alpha\trinl\boldsymbol{\xi}\trinr_{2+\alpha}\trinl\Phi_{\dg}\trinr_{\dg}.
	\end{equation}
\end{lem}

\begin{proof}
	The method of real interpolation \cite[p. 374]{Brenner} of Sobolev spaces defines $H^{2+\alpha}(\Omega):=\left[ H^2(\Omega),H^3(\Omega)\right]_{\alpha,2}$ from $H^2(\Omega)$ and $H^3(\Omega)$.
	Given any $\xi\in H^2(\Omega)$ and $\varphi_{\dg}\in \St$, the boundedness of $a_{\dg}(\bullet,\bullet)$ and Lemma~\ref{enrichment_dG} imply
	\begin{align}\label{H2a}
	a_{\dg}(\Pi_h\xi,\varphi_{\dg}-E_h\varphi_{\dg})\lesssim\|\Pi_h\xi\|_{\dg}\|\varphi_{\dg}-E_h\varphi_{\dg}\|_{\dg}\lesssim \|\xi\|_2\|\varphi_{\dg}\|_{\dg}.
	\end{align}
	Given any $\xi\in H^3(\Omega)$ and $\varphi_{\dg}\in \St$, the definition of $a_{\dg}(\bullet,\bullet)$, an integration by parts, and Lemma~\ref{enrichment_dG} result in
	\begin{align*}
	a_{\dg}(\xi,\varphi_{\dg}-E_h\varphi_{\dg})
	&=\sum_{K\in\cT}\int_K D^2\xi:D^2(\varphi_{\dg}-E_h\varphi_{\dg})\dx- J(\xi,\varphi_{\dg}-E_h\varphi_{\dg})\notag\\
	&=\sum_{K\in\cT}\int_K\divc(D^2\xi)\cdot\nabla(\varphi_{\dg}-E_h\varphi_{\dg})\dx\leq Ch\|\xi\|_3\|\varphi_{\dg}\|_{\dg}.
	\end{align*}
	This,  Remark \ref{a_bdd_h}, \eqref{interpolation_est} with $\alpha=1$, and Lemma~\ref{enrichment_dG} lead to
	\begin{align}\label{H3a3}
	a_{\dg}(\Pi_h\xi,\varphi_{\dg}-E_h\varphi_{\dg})&=a_{\dg}(\Pi_h\xi-\xi,\varphi_{\dg}-E_h\varphi_{\dg})+a_{\dg}(\xi,\varphi_{\dg}-E_h\varphi_{\dg}) \nonumber \\
	& \lesssim h\|\xi\|_3\|\varphi_{\dg}\|_{\dg}.
	\end{align}
	The estimates \eqref{H2a}-\eqref{H3a3} and the interpolation between Sobolev spaces show
	\begin{equation*}
	a_{\dg}(\Pi_h\xi,\varphi_{\dg}-E_h\varphi_{\dg})\lesssim h^\alpha \|\xi\|_{2+\alpha}\|\varphi_{\dg}\|_{\dg}.
	\end{equation*}
	\noindent For any $\alpha>1/2$, the combination with Remark~\ref{a_bdd_h}, \eqref{interpolation_est}, and Lemma~\ref{enrichment_dG} results in
	\begin{align*}
	a_{\dg}(\xi,\varphi_{\dg}-E_h\varphi_{\dg})& =a_{\dg}(\xi-\Pi_h\xi,\varphi_{\dg}-E_h\varphi_{\dg})+a_{\dg}(\Pi_h\xi,\varphi_{\dg}-E_h\varphi_{\dg}) \\
	& \lesssim h^\alpha\|\xi\|_{2+\alpha}\|\varphi_{\dg}\|_{\dg}. 
	\end{align*}
	\noindent This concludes the proof of \eqref{Inpol_BTW}. The estimate \eqref{Inpol_BTW_A} follows from \eqref{Inpol_BTW} and the definition of $A_{\dg}(\bullet,\bullet)$.
\end{proof}


\begin{lem}[Boundedness of $b_{\dg}(\bullet, \bullet,\bullet)$]\label{b_dG_1} 
	\begin{enumerate}[(a)]
		\item Any  $\eta,\chi, \varphi\in X+\St$ satisfy $$\displaystyle b_{\dg}(\eta,\chi,\varphi)\lesssim \|\eta\|_{\dg}\|\chi\|_{\dg}\|\varphi\|_{\dg}.$$
		\item Given any $\alpha>1/2$, any $\eta\in X\cap H^{2+\alpha}(\Omega)$ and $\chi\in X+\St$ satisfy
		\begin{align*}
		b_{\dg}(\eta,\chi,\varphi)\lesssim
		\begin{cases}
		\|\eta\|_{2+\alpha}\|\chi\|_{\dg}\|\varphi\|_1\fl \varphi\in H^1_0(\Omega),\\
		\|\eta\|_{2+\alpha}\|\chi\|_{\dg}\|\varphi\|_{L^4(\Omega)}\fl\varphi\in X+\St.
		\end{cases}
		\end{align*}
	\end{enumerate}	
\end{lem}

\medskip
\begin{proof}
	\noindent{\it (a).}
	For $\eta,\chi,\varphi \in X+\St$, the definition of $b_{\dg}(\bullet,\bullet,\bullet)$ and Lemma~\ref{auxest} lead to
	\begin{align*}
	{|}2b_{\dg}(\eta,\chi,\varphi){|}&={|}\sik [\eta,\chi]\varphi\dx{|}\lesssim |\eta|_{ H^2(\cT)}|\chi|_{ H^2(\cT)}\|\eta\|_{L^{\infty}(\cT)}\\
	&\lesssim\|\eta\|_{\dg}\|\chi\|_{\dg}\|\varphi\|_{\dg}.
	\end{align*}

	\noindent{\it (b).} For $\eta\in X\cap H^{2+\alpha}(\Omega)$, $\chi\in X+\St$, and $\varphi\in H^1_0(\Omega)\cup (X+\St)$, the generalized \Holder inequality and the continuous imbedding $H^{2+\alpha}(\Omega)\hookrightarrow W^{2,4}(\Omega)$ yields
	\begin{align*}
	{|}2b_{\dg}(\eta,\chi,\varphi){|}&={|}\sik [\eta,\chi]\varphi\dx{|}
	\lesssim \|\eta\|_{W^{2,4}(\Omega)}\|\chi\|_{H^{2}(\cT)}\|\varphi\|_{L^4(\Omega)}\\
	&\lesssim\|\eta\|_{2+\alpha}\|\chi\|_{\dg}\|\varphi\|_{L^4(\Omega)}.
	\end{align*}
	This implies second part of (b). For $\varphi\in H^{1}_0(\Omega)\hookrightarrow L^{4}(\Omega)$ this proves the first.
\end{proof}

\section{{\it A priori} error control}

This section first establishes the discrete inf-sup condition for the linearized problem, then the existence of a discrete solution to the nonlinear problem \eqref{vform_d_dG}, and finally the convergence of a Newton method. 

\subsection{Discrete inf-sup condition}
This subsection is devoted to discrete inf-sup condition. Throughout the paper, the statement {\it ``there exists $\ubar{\sigma}_2$ such that for all $\sigma_2\geq \ubar{\sigma}_2$ as in Lemma~\ref{ellipticity} on ellipticity, there exists $h(\sigma_2)$ such that for all $h\leq h(\sigma_2)$ ...''} is abbreviated by the phrase {\it ``for sufficiently large $\sigma_2$ and sufficiently small $h$ ...''.}

\begin{thm}[Discrete inf-sup condition]\label{dis_inf_sup_thm} Let $\Psi\in \bH^{2+\alpha}(\Omega)\cap \bH^2_0(\Omega)$ be a regular solution to \eqref{vform_cts_dG}. For sufficiently large $\sigma_2$ and sufficiently small $h$, there exists $\widehat{\beta}$ such that the following discrete inf-sup condition holds
	\begin{align}\label{dis_inf_sup_dG}
	0<\widehat{\beta}\leq \inf_{\substack{\Theta_{\dg}\in \Stb\\ \trinl\Theta_{\dg}\trinr_{\dg}=1}}\sup_{\substack{\Phi_{\dg}\in \Stb\\ \trinl\Phi_{\dg}\trinr_{\dg}=1}}\Big{(}A_{\dg}(\Theta_{\dg},\Phi_{\dg})+2B_{\dg}(\Psi,\Theta_{\dg},\Phi_{\dg})\Big{)}.
	\end{align}
\end{thm}
\begin{proof}
	Given any $\Theta_{\dg}\in\Stb$ with $\trinl\Theta_{\dg}\trinr_{\dg}=1$, let $\boldsymbol{\xi}\in \bX$ and $\boldsymbol{\eta}\in \bX$ solve the biharmonic problems
	\begin{align}
	A(\boldsymbol{\xi},\Phi)&=2B_{\dg}(\Psi,\Theta_{\dg},\Phi)\fl\Phi\in \bX,\label{defn_xi_dG}\\
	A(\boldsymbol{\eta},\Phi)&=2B(\Psi,E_h\Theta_{\dg},\Phi)\fl\Phi\in \bX.\label{defn_Eta_dG}
	\end{align}
	Lemma~\ref{b_dG_1}.b implies that $B_{\dg}(\Psi,\tilde{\Theta},\bullet)$ and $B (\Psi,\tilde{\Theta},\bullet)$ belong to $\bH^{-1}(\Omega)$ for $\tilde{\Theta}\in \bX+\Stb$. The reduced elliptic regularity for the biharmonic problem \cite{BlumRannacher} yields $\boldsymbol{\xi},\boldsymbol{\eta}\in \bH^{2+\alpha}(\Omega)\cap \bX$. Since $\Psi$ is a regular solution to \eqref{vform_cts_dG}, there exists $\beta$ from \eqref{inf-sup_dG} and $\Phi\in \bX$ with $\trinl\Phi\trinr_2=1$ such that
	\begin{align*}
	\beta\trinl E_h\Theta_{\dg}\trinr_2\leq A(E_h\Theta_{\dg},\Phi)+2B(\Psi,E_h\Theta_{\dg},\Phi).
	\end{align*}
	The solution property in \eqref{defn_Eta_dG}, the boundedness of $A(\bullet,\bullet)$, and the triangle inequality in the above result imply
	\begin{align*} 
	\beta\trinl E_h\Theta_{\dg}\trinr_2 & \leq A(E_h\Theta_{\dg}+\boldsymbol{\eta},\Phi)\leq \trinl E_h\Theta_{\dg}+\boldsymbol{\eta}\trinr_2 \nonumber \\
	& \leq \trinl  E_h\Theta_{\dg}-\Theta_{\dg}\trinr_{\dg}+\trinl\Theta_{\dg}+\boldsymbol{\xi}\trinr_{\dg}+\trinl\boldsymbol{\eta}-\boldsymbol{\xi}\trinr_2.
	\end{align*}
	The definition of $\boldsymbol{\xi},\boldsymbol{\eta}$ in \eqref{defn_xi_dG}-\eqref{defn_Eta_dG} and Lemma~\ref{b_dG_1}.a lead to
	\begin{equation*} 
	\trinl\boldsymbol{\eta}-\boldsymbol{\xi}\trinr_2\leq 2C_b\trinl\Psi\trinr_{2}\trinl E_h\Theta_{\dg}-\Theta_{\dg}\trinr_{\dg}
	\end{equation*}
	for some positive constant $C_b\approx 1$.
	The combination of the previous two displayed inequalities reads
	\begin{equation*}
	\beta\trinl E_h\Theta_{\dg}\trinr_2\leq \trinl\Theta_{\dg}+\boldsymbol{\xi}\trinr_{\dg}+\left(1+2C_b\trinl\Psi\trinr_2\right)\trinl E_h\Theta_{\dg}-\Theta_{\dg}\trinr_{\dg}.
	\end{equation*}
	This and \eqref{EhLambda} result in 
	\begin{equation}\label{theta_final_bdd_dG}
	\trinl E_h\Theta_{\dg}\trinr_2\leq \frac{1}{\beta}\Big{(}1+\Lambda\left(1+2C_b\trinl\Psi\trinr_2\right)\Big{)}\trinl\Theta_{\dg}+\boldsymbol{\xi}\trinr_{\dg}.
	\end{equation}
	The triangle inequality, \eqref{theta_final_bdd_dG}, and \eqref{EhLambda} lead to
	\begin{align*}
	1&=\trinl\Theta_{\dg}\trinr_{\dg}\leq \trinl E_h\Theta_{\dg}-\Theta_{\dg}\trinr_{\dg}+\trinl E_h\Theta_{\dg}\trinr_2\notag\\
	&\leq\left(\Lambda+\frac{1}{\beta}\Big{(}1+\Lambda\left(1+2C_b\trinl\Psi\trinr_2\right)\Big{)}\right)\trinl\Theta_{\dg}+\boldsymbol{\xi}\trinr_{\dg}.
	\end{align*}
	In other words, $\displaystyle \beta_1:=\beta/\left(1+\Lambda\left(1+\beta+2C_b\trinl\Psi\trinr_2\right)\right)$ satisfies
	\begin{equation}
	 \beta_1\leq \trinl\Theta_{\dg}+\boldsymbol{\xi}\trinr_{\dg}\leq \trinl\Theta_{\dg}+\Pi_h\boldsymbol{\xi}\trinr_{\dg}+\trinl \boldsymbol{\xi}-\Pi_h\boldsymbol{\xi}\trinr_{\dg}.\label{betak_bdd_dG}
	\end{equation}
	For any given $\Theta_{\dg}+\Pi_h\boldsymbol{\xi}\in \Stb$, the ellipticity of $A_{\dg}(\bullet,\bullet)$ from Lemma~\ref{ellipticity} implies the existence of some $\Phi_{\dg}\in\Stb$ with $\trinl\Phi_{\dg}\trinr_{\dg}=1$ and
	\begin{align*}
	& \beta_0\trinl\Theta_{\dg}+\Pi_h\boldsymbol{\xi}\trinr_{\dg} \leq A_{\dg}(\Theta_{\dg}+\Pi_h\boldsymbol{\xi},\Phi_{\dg}) \notag\\ 
	&\qquad =A_{\dg}(\Theta_{\dg},\Phi_{\dg})+A_{\dg}(\Pi_h\boldsymbol{\xi}-\boldsymbol{\xi},\Phi_{\dg})+A_{\dg}(\boldsymbol{\xi},\Phi_{\dg}-E_h\Phi_{\dg})+A(\boldsymbol{\xi},E_h\Phi_{\dg}).
	\end{align*}
	The choice of $\Phi=E_h\Phi_{\dg}$ in \eqref{defn_xi_dG} plus straightforward calculations result in
	\begin{align}
	\beta_0\trinl\Theta_{\dg}+\Pi_h\boldsymbol{\xi}\trinr_{\dg}
	& \leq A_{\dg}(\Theta_{\dg},\Phi_{\dg})+ A_{\dg}(\Pi_h\boldsymbol{\xi}-\boldsymbol{\xi},\Phi_{\dg}) 
	+A_{\dg}(\boldsymbol{\xi},\Phi_{\dg}-E_h\Phi_{\dg}) \notag \\
	& \; \; + 2B_{\dg}(\Psi,\Theta_{\dg},\Phi_{\dg})
	+2B_{\dg}(\Psi,\Theta_{\dg},E_h\Phi_{\dg}-\Phi_{\dg}). \label{final_inf-sup_bdd_dG}
	\end{align}
	\noindent Remark~\ref{a_bdd_h}, Lemma \ref{interpolant_dG} and \ref{intermediate} for the second and third term, plus Lemma~\ref{b_dG_1}.b and \ref{enrichment_dG} for the last term in \eqref{final_inf-sup_bdd_dG} lead to 	
	\begin{equation}\label{pre_inf-sup_bdd_dG}
	\beta_0\trinl\Theta_{\dg}+\Pi_h\boldsymbol{\xi}\trinr_{\dg}\leq A_{\dg}(\Theta_{\dg},\Phi_{\dg})+2B_{\dg}(\Psi,\Theta_{\dg},\Phi_{\dg})+Ch^{\alpha}.
	\end{equation}
	The combination of \eqref{betak_bdd_dG} and \eqref{pre_inf-sup_bdd_dG} with Lemma~\ref{interpolant_dG} shows
	\begin{align*}
	\beta_0\beta_1-C_*h^\alpha\leq A_{\dg}(\Theta_{\dg},\Phi_{\dg})+2B_{\dg}(\Psi,\Theta_{\dg},\Phi_{\dg})
	\end{align*}
	for some $h$-independent positive constant $C_*$. Hence,  for all  $h\leq h_0:=(\beta_0\beta_1/2C_*)^{\frac{1}{\alpha}}$, the discrete inf-sup condition \eqref{dis_inf_sup_dG} follows.
\end{proof}

\noindent The next lemma establishes that the perturbed bilinear form  
\begin{equation}\label{defnAht_dG}
\tilde \cA_{\dg}(\Theta_{\dg},\Phi_{\dg}):= A_{\dg}(\Theta_{\dg},\Phi_{\dg})+2B_{\dg}(\Pi_h\Psi,\Theta_{\dg},\Phi_{\dg})
\end{equation}
satisfies a discrete inf-sup condition.

\begin{lem}\label{regular_dG} 
    Let $\Pi_h\Psi$ be the interpolation of $\Psi$ from Lemma~\ref{interpolant_dG}.
	Then, for sufficiently large $\sigma_2$ and sufficiently small $h$, the perturbed bilinear form \eqref{defnAht_dG} satisfies the discrete inf-sup condition
	\begin{align}\label{dis_inf_sup_dG_perturb}
	\frac{\widehat{\beta}}{2}\leq \inf_{\substack{\Theta_{\dg}\in \Stb\\ \trinl\Theta_{\dg}\trinr_{\dg}=1}}\sup_{\substack{\Phi_{\dg}\in \Stb\\ \trinl\Phi_{\dg}\trinr_{\dg}=1}}\tilde \cA_{\dg}(\Theta_{\dg},\Phi_{\dg}).
	\end{align}
\end{lem}
{\it Proof.}
	Lemma~\ref{interpolant_dG} leads to the existence of $h_1>0$, such that  $\trinl\Psi-\Pi_h\Psi\trinr_{\dg}\leq\widehat{\beta}/4C_b$ holds for $h\leq h_1$. Given any $\Theta_{\dg}\in\St$,
	Theorem~\ref{dis_inf_sup_thm} and  Lemma~\ref{b_dG_1}.a lead to 
	\begin{align*}
	&\frac{\widehat{\beta}}{2}\trinl\Theta_{\dg}\trinr_{\dg}\leq \widehat{\beta}\trinl\Theta_{\dg}\trinr_{\dg}-2C_b\trinl\Psi-\Pi_h\Psi\trinr_{\dg}\trinl\Theta\trinr_{\dg}\\
	&\leq \sup_{\substack{\Phi_{\dg}\in \Stb\\ \trinl\Phi_{\dg}\trinr_{\dg}=1}} \left(A_{\dg}(\Theta_{\dg},\Phi_{\dg})+2B_{\dg}(\Psi,\Theta_{\dg},\Phi_{\dg})\right) \\
	& \qquad\qquad\qquad-2\sup_{\substack{\Phi_{\dg}\in \Stb\\ \trinl\Phi_{\dg}\trinr_{\dg}=1}} B_{\dg}(\Psi-\Pi_h\Psi,\Theta_{\dg},\Phi_{\dg})\\
	&\leq\sup_{\substack{\Phi_{\dg}\in \Stb\\ \trinl\Phi_{\dg}\trinr_{\dg}=1}} \left(A_{\dg}(\Theta_{\dg},\Phi_{\dg})+2B_{\dg}(\Pi_h\Psi,\Theta_{\dg},\Phi_{\dg})\right)=\sup_{\substack{\Phi_{\dg}\in \Stb\\ \trinl\Phi_{\dg}\trinr_{\dg}=1}}\tilde{\cA}_{\dg}(\Theta_{\dg},\Phi_{\dg}).\qed
	\end{align*}

\subsection{Existence, uniqueness and error estimate}
The discrete inf-sup condition is employed to define a nonlinear map $\mu:\Stb\to\Stb$ which enables to analyze the existence and uniqueness of solution of \eqref{vform_d_dG}. 
For any $\Theta_{\dg}\in\Stb$, define $\mu(\Theta_{\dg})\in\Stb$ as the solution to the discrete fourth-order problem
\begin{align}\label{defnmu_dG}
\tilde{\cA}_{\dg}(\mu(\Theta_{\dg}),\Phi_{\dg})=L_{\dg}(\Phi_{\dg})+2B_{\dg}(\Pi_h\Psi,\Theta_{\dg},\Phi_{\dg})-B_{\dg}(\Theta_{\dg},\Theta_{\dg},\Phi_{\dg})
\end{align} 
for all $\Phi_{\dg}\in\Stb$.
Lemma~\ref{regular_dG} guarantees that the mapping $\mu$ is well-defined and continuous. Also, any fixed point of $\mu$ is a solution to \eqref{vform_d_dG} and vice-versa. In order to show that the mapping $\mu$ has a fixed point, define the ball \[\bB_{R}(\Pi_h\Psi):=\Big{\{}\Phi_{\dg} \in\Stb: \trinl\Phi_{\dg}-\Pi_h\Psi\trinr_{\dg}\leq R \Big{\}}.\]

\begin{thm}[Mapping of ball to ball]\label{mapball2ball}
	For sufficiently large $\sigma_2$ and sufficiently small $h$, there exists a positive constant $R(h)$ such that $\mu$ maps the ball $\bB_{R(h)}(\Pi_h\Psi)$ to itself; $\displaystyle\trinl\mu(\Theta_{\dg} )-\Pi_h\Psi\trinr_{\dg}\leq R(h)$ holds for any $\Theta_{\dg}  \in \bB_{R(h)}(\Pi_h\Psi)$.
\end{thm}

\begin{proof}
	The discrete inf-sup condition of $\tilde{\cA}_{\dg}(\bullet,\bullet)$ in Lemma \ref{regular_dG} implies the existence of $\Phi_{\dg} \in\Stb$ with  $\trinl\Phi_{\dg} \trinr_{\dg}=1$ and 
	\begin{align*}
	&\frac{\widehat{\beta}}{2}\trinl\mu(\Theta_{\dg})-\Pi_h\Psi\trinr_{\dg}\leq \tilde\cA_{\dg}(\mu(\Theta_{\dg})-\Pi_h\Psi, \Phi_{\dg}).
	\end{align*}
	Let $E_h\Phi_{\dg}$ be the enrichment of $\Phi_{\dg}$ from Lemma~\ref{enrichment_dG}.   The definition of $\tilde{\cA}_{\dg}(\bullet,\bullet)$, the symmetry of $B_{\dg}(\bullet,\bullet,\bullet)$ in first and second variables, \eqref{defnmu_dG}, and \eqref{vform_cts_dG} lead to 
	\begin{align}
	&\tilde\cA_{\dg}(\mu(\Theta_{\dg})-\Pi_h\Psi, \Phi_{\dg})
	=\tilde\cA_{\dg}(\mu(\Theta_{\dg}), \Phi_{\dg})-\tilde\cA_{\dg}(\Pi_h\Psi, \Phi_{\dg})\notag\\
	&\quad=L_{\dg}(\Phi_{\dg})+2B_{\dg}(\Pi_h\Psi,\Theta_{\dg},\Phi_{\dg})-B_{\dg}(\Theta_{\dg},\Theta_{\dg},\Phi_{\dg})-A_{\dg}(\Pi_h\Psi,\Phi_{\dg})\nonumber \\
	& \qquad \; -2B_{\dg}(\Pi_h\Psi,\Pi_h\Psi,\Phi_{\dg})\notag\\
	&\quad=L_{\dg}(\Phi_{\dg}-E_h\Phi_{\dg}) +\left(A(\Psi, E_h\Phi_{\dg})-A_{\dg}(\Pi_h\Psi,\Phi_{\dg})\right) \notag \\
	& \qquad \; +\left(B(\Psi,\Psi, E_h\Phi_{\dg})-B_{\dg}(\Pi_h\Psi,\Pi_h\Psi,\Phi_{\dg})\right)
	+B_{\dg}(\Pi_h\Psi-\Theta_{\dg},\Theta_{\dg}-\Pi_h\Psi,\Phi_{\dg})\notag \\ & \quad =:T_1+T_2+T_3+T_4. \label{Aht_Est}
	\end{align}
	The term $T_1$ can be estimated using the continuity of $L_{\dg}$ and Lemma~\ref{enrichment_dG}.
	{The continuity of $A_{\dg}(\bullet, \bullet)$, Lemma~\ref{intermediate} and \ref{interpolant_dG} with $\trinl\Phi_{\dg} \trinr_{\dg}=1$  lead to}
	\begin{align*}
	T_2:&=A(\Psi, E_h\Phi_{\dg})-A_{\dg}(\Pi_h\Psi,\Phi_{\dg})\\
	 & = A_{\dg}(\Psi,E_h\Phi_{\dg}-\Phi_{\dg})
	+A_{\dg}(\Psi-\Pi_h\Psi,\Phi_{\dg}) \lesssim h^\alpha \trinl\Psi\trinr_{2+\alpha}.
	\end{align*}
	Lemma~\ref{b_dG_1}, \ref{enrichment_dG}, and \ref{interpolant_dG} result in 
	\begin{align*}
	T_3&:= B(\Psi,\Psi,E_h\Phi_{\dg})-B_{\dg}(\Pi_h\Psi,\Pi_h\Psi,\Phi_{\dg})\\
	&= B_{\dg}(\Psi,\Psi-\Pi_h\Psi,E_h\Phi_{\dg})+B_{\dg}(\Psi-\Pi_h\Psi,\Pi_h\Psi,\Phi_{\dg})
	\\
	&\quad+B_{\dg}(\Psi,\Pi_h\Psi,E_h\Phi_{\dg}-\Phi_{\dg}) \lesssim h^\alpha\trinl\Psi\trinr_{2+\alpha}\trinl\Psi\trinr_{2}.
	\end{align*} 
	Lemma~\ref{b_dG_1} implies
	\begin{align*}
	T_4:=B_{\dg}(\Pi_h\Psi-\Theta_{\dg},\Theta_{\dg}-\Pi_h\Psi,\Phi_{\dg})\lesssim \trinl\Theta_{\dg}-\Pi_h\Psi\trinr_{\dg}^2.
	\end{align*}
	A substitution of the estimates for $T_1, T_2, T_3$, and $T_4$ in \eqref{Aht_Est} and $\trinl\Psi\trinr_{2+\alpha}\approx 1\approx \trinl\Psi\trinr_2$ lead to $C_1\approx 1$ with
	\begin{align}\label{muestimate_dG}
	&\trinl\mu(\Theta_{\dg})-\Pi_h\Psi\trinr_{\dg}\leq C_1\left( h^{\alpha}+\trinl\Theta_{\dg}-\Pi_h\Psi\trinr_{\dg}^2\right).
	\end{align}
	Then $h\leq h_2:=\left(2C_1\right)^{-2/\alpha}$ and $\trinl \Theta_{\dg}-\Pi_h\Psi\trinr_{\dg}\leq R(h):=2C_1 h^\alpha$ lead to
	\begin{align*}
	&\trinl\mu(\Theta_{\dg})-\Pi_h\Psi\trinr_{\dg}\leq C_1 h^\alpha\left(1+4C_1^2 h^\alpha\right)\leq R(h).
	\end{align*}
	This concludes the proof.
\end{proof}

\begin{thm}[Existence and uniqueness]\label{contractionthm}
For sufficiently large $\sigma_2$ and sufficiently small $h$, there exists a unique solution $\Psi_{\dg}$ to the discrete problem \eqref{vform_d_dG} in $\bB_{R(h)}(\Pi_h\Psi)$.
\end{thm}

\begin{proof}

	First we prove the contraction result of the nonlinear map $\mu$ in the ball $\bB_{R(h)}(\Pi_h\Psi)$ of Theorem~\ref{mapball2ball}. Given any $\Theta_{\dg},\tilde{\Theta}_{\dg}\in  \bB_{R(h)}(\Pi_h\Psi)$ and for all $\Phi_{\dg}\in\Stb$, the solutions $\mu(\Theta_{\dg})$ and  $\mu(\tilde{\Theta}_{\dg})$ satisfy
	\begin{align}
	\tilde\cA_{\dg}(\mu(\Theta_{\dg}),\Phi_{\dg})&=L_{\dg}(\Phi_{\dg})+2B_{\dg}(\Pi_h\Psi,\Theta_{\dg},\Phi_{\dg})-B_{\dg}(\Theta_{\dg},\Theta_{\dg},\Phi_{\dg}), \label{mueqn1_dG1} \\
	\tilde\cA_{\dg}(\mu(\tilde{\Theta}_{\dg}),\Phi_{\dg})&=L_{\dg}(\Phi_{\dg})+2B_{\dg}(\Pi_h\Psi,\tilde{\Theta}_{\dg},\Phi_{\dg})-B_{\dg}(\tilde{\Theta}_{\dg},\tilde{\Theta}_{\dg},\Phi_{\dg}). \label{mueqn1_dG2}
	\end{align}
	\noindent 	The discrete inf-sup of $\tilde\cA_{\dg}(\bullet,\bullet)$ from Lemma~\ref{regular_dG} guarantees the existence of $\Phi_{\dg}$ with $\trinl \Phi_{\dg}\trinr_{\dg}=1$ below.  With \eqref{mueqn1_dG1}-\eqref{mueqn1_dG2} and Lemma~\ref{b_dG_1}, it follows that
	\begin{align*}
	&\frac{\widehat{\beta}}{2}\trinl\mu(\Theta_{\dg})-\mu(\tilde{\Theta}_{\dg})\trinr_{\dg}\leq \tilde\cA_{\dg}(\mu(\Theta_{\dg})-\mu(\tilde{\Theta}_{\dg}),\Phi_{\dg})\notag\\
	&=2B_{\dg}(\Pi_h\Psi,\Theta_{\dg}-\tilde{\Theta}_{\dg}, \Phi_{\dg})+B_{\dg}(\tilde{\Theta}_{\dg},\tilde{\Theta}_{\dg}, \Phi_{\dg})-B_{\dg}(\Theta_{\dg},\Theta_{\dg}, \Phi_{\dg})\\
	&=B_{\dg}(\tilde{\Theta}_{\dg}-\Theta_{\dg},\Theta_{\dg}-\Pi_h\Psi,\Phi_{\dg})+B_{\dg}(\tilde{\Theta}_{\dg}-\Pi_h\Psi,\tilde{\Theta}_{\dg}-\Theta_{\dg},\Phi_{\dg})\\
	&\lesssim\trinl\tilde{\Theta}_{\dg}-\Theta_{\dg}\trinr_{\dg}\left(\trinl\Theta_{\dg}-\Pi_h\Psi\trinr_{\dg}+
	\trinl\tilde{\Theta}_{\dg}-\Pi_h\Psi\trinr_{\dg}\right).
	\end{align*}
	Since $\Theta_{\dg},\tilde{\Theta}_{\dg}\in  \bB_{R(h)}(\Pi_h\Psi)$, for a choice of $R(h)$ as in the proof of  Theorem \ref{mapball2ball}, for sufficiently large $\sigma_2$ and $h\leq \min\{h_0,h_1,h_2\}$, 
	\begin{equation*}
	\trinl\mu(\Theta_{\dg})-\mu(\tilde{\Theta}_{\dg})\trinr_{\dg}\lesssim h^{\alpha} \trinl\tilde{\Theta}_{\dg}-\Theta_{\dg}\trinr_{\dg}.
	\end{equation*}
	Hence, there exists positive constant $h_3$, such that for $h\leq h_3$ the contraction result holds.
	
	\medskip
	
	For $h\leq \ubar{h}:=\min\{h_0,h_1,h_2,h_3\}$, Lemma~\ref{mapball2ball} and Theorem~\ref{contractionthm} lead to the fact that $\mu$ is a contraction map that maps the ball $\bB_{R(h)}(\Pi_h\Psi)$ into itself. {An application of the Banach fixed point theorem} yields that the mapping $\mu$ has a unique fixed point in the ball $\bB_{R(h)}(\Pi_h\Psi)$, say $\Psi_{\dg}$, which solves \eqref{vform_d_dG} with $\trinl\Psi_{\dg}-\Pi_h\Psi\trinr_{\dg}\leq R(h)$.
\end{proof}

\begin{thm}[Energy norm estimate]\label{eetimate}
	Let  $\Psi$ be a regular solution to \eqref{vform_cts_dG}  and let $\Psi_{\dg}$ be the solution to \eqref{vform_d_dG}. For sufficiently large $\sigma_2$ and sufficiently small $h$, it holds
	\begin{equation*}
	\trinl\Psi-\Psi_{\dg}\trinr_{\dg}\leq C h^{\alpha}.
	\end{equation*}
\end{thm}

\begin{proof}
	A triangle inequality yields
	\begin{equation}\label{newnn_dG}
	\trinl\Psi-\Psi_{\dg}\trinr_{\dg}\leq \trinl\Psi-\Pi_h\Psi\trinr_{\dg}+\trinl\Pi_h\Psi-\Psi_{\dg}\trinr_{\dg}.
	\end{equation}
	For $h\leq \ubar{h}$ and sufficiently large $\sigma_2$, Theorem~\ref{contractionthm} leads to
	\begin{equation}\label{pihsoln}
	\trinl\Pi_h\Psi-\Psi_{\dg}\trinr_{\dg}\leq Ch^\alpha.
	\end{equation}
	This, Lemma~\ref{interpolant_dG}, \eqref{pihsoln},  and \eqref{newnn_dG} conclude the proof.
\end{proof}

\subsection{Convergence of the Newton method}

The discrete solution $\Psi_{\dg}$ of \eqref{vform_d_dG} is characterized as the fixed point of \eqref{defnmu_dG} and so depends on the unknown $\Pi_h\Psi$. The approximate solution to \eqref{vform_d_dG} is computed with the Newton method, where the iterates $\Psi_{\dg}^{j}$ solve 
\begin{equation}\label{NewtonIterate_dG}
A_{\dg}(\Psi_{\dg}^{j},\Phi_{\dg})+2B_{\dg}(\Psi_{\dg}^{j-1},\Psi_{\dg}^{j},\Phi_{\dg})=B_{\dg}(\Psi_{\dg}^{j-1},\Psi_{\dg}^{j-1},\Phi_{\dg})+L_{\dg}(\Phi_{\dg})
\end{equation} 
for all $\Phi_{\dg}\in \Stb$. The Newton method has locally quadratic convergence.
\begin{thm}[Convergence of Newton method]\label{NewtonThm}
	Let $\Psi$ be a regular solution to \eqref{vform_cts_dG} and let  $\Psi_{\dg}$ solve \eqref{vform_d_dG}. There exists a positive constant $R$ independent of $h$, such that for any initial guess $\Psi_{\dg}^0$ with
	$\displaystyle \trinlNH \Psi_{\dg}- \Psi_{\dg}^0\trinrNH_{\dg}\leq  R$, it follows $\trinlNH\Psi_{\dg}- \Psi_{\dg}^{j}\trinrNH_{\dg}  \leq R\fl j=0,1,2,\ldots$ and the iterates of the Newton method  in \eqref{NewtonIterate_dG} are well defined and converges quadratically to $\Psi_{\dg}$.
\end{thm}
\begin{proof}
	Following the proof of Lemma~\ref{regular_dG}, there exists a positive constant $\epsilon$ (sufficiently small) independent of $h$ such that for each $Z_{\dg}\in\Stb$ with $\trinlNH Z_{\dg}-\Pi_h\Psi\trinrNH_{\dg}\leq \epsilon$, the bilinear form
	\begin{equation}\label{NewtonNonsingular}
	A_{\dg}(\bullet,\bullet)+2B_{\dg}(Z_{\dg},\bullet,\bullet)
	\end{equation}
	\smallskip
	satisfies discrete inf-sup condition in $\Stb\times\Stb$.
	 For sufficiently large $\sigma_2$ and sufficiently small $h$, the equation \eqref{pihsoln} implies $\trinlNH\Pi_h\Psi-\Psi_{\dg}\trinrNH_{\dg}\leq Ch^\alpha$. Thus $h$ can be chosen sufficiently small so that $\trinlNH\Pi_h\Psi-\Psi_{\dg}\trinrNH_{\dg}\leq \epsilon/2$. Recall $\widehat{\beta}$ from \eqref{dis_inf_sup_dG}. Lemma \ref{b_dG_1}.a implies that there exists a positive constant $C_{\tilde{b}}$ independent of $h$ such that  $B_{\dg}(\Xi_{\dg},\Theta_{\dg},\Phi_{\dg})\leq C_{\tilde{b}}\trinl\Xi_{\dg}\trinr_{\dg}\trinl\Theta_{\dg}\trinr_{\dg}\trinl\Phi_{\dg}\trinr_{\dg}$. Set
	\begin{equation*}
	 R:=\min\left\{\epsilon/2,\widehat{\beta}/8C_{\tilde{b}}\right\}.
	\end{equation*}
	Assume that the initial guess $\Psi_{\dg}^0$ satisfies $\trinlNH\Psi_{\dg}-\Psi_{\dg}^0\trinrNH_{\dg}\leq  R$. Then,
	\begin{equation*}
	\trinlNH\Pi_h\Psi-\Psi_{\dg}^0\trinrNH_{\dg}\leq \trinlNH\Pi_h\Psi-\Psi_{\dg}\trinrNH_{\dg}+\trinlNH\Psi_{\dg}-\Psi_{\dg}^0\trinrNH_{\dg}\leq \epsilon.
	\end{equation*}
	This implies $\trinlNH\Psi_{\dg}-\Psi_{\dg}^{j-1}\trinrNH_{\dg}\leq  R$ for $j=1$ and suppose for mathematical induction that this holds for some $j\in\bN$. Then $Z_{\dg}:=\Psi_{\dg}^{j-1}$ in \eqref{NewtonNonsingular} leads to an discrete inf-sup condition of $A_{\dg}(\bullet,\bullet)+2B_{\dg}(\Psi^{j-1}_{\dg},\bullet,\bullet)$ and so to an unique solution $\Psi_{\dg}^j$ in step $j$ of the Newton scheme. The discrete inf-sup condition \eqref{NewtonNonsingular} implies the existence of $ \Phi_{\dg}\in\Stb$ with $\trinlNH \Phi_{\dg}\trinrNH_{\dg}=1$ and
	\begin{equation*}
	\frac{\widehat{\beta}}{4}\trinlNH\Psi_{\dg}-\Psi_{\dg}^{j}\trinrNH_{\dg}\leq A_{\dg}(\Psi_{\dg}-\Psi_{\dg}^{j}, \Phi_{\dg})+2B_{\dg}(\Psi_{\dg}^{j-1},\Psi_{\dg}-\Psi_{\dg}^{j}, \Phi_{\dg}).
	\end{equation*}
	The application of \eqref{NewtonIterate_dG}, \eqref{vform_d_dG}, and Lemma~\ref{b_dG_1} result in 
	\begin{align}
	&A_{\dg}(\Psi_{\dg}-\Psi_{\dg}^{j}, \Phi_{\dg})+2B_{\dg}(\Psi_{\dg}^{j-1},\Psi_{\dg}-\Psi_{\dg}^{j}, \Phi_{\dg})\notag\\
	&=A_{\dg}(\Psi_{\dg},  \Phi_{\dg})+2B_{\dg}(\Psi_{\dg}^{j-1},\Psi_{\dg},  \Phi_{\dg})-B_{\dg}(\Psi_{\dg}^{j-1},\Psi_{\dg}^{j-1},  \Phi_{\dg})-L_{\dg}(  \Phi_{\dg})\notag\\
	&=-B_{\dg}(\Psi_{\dg},\Psi_{\dg},\Phi_{\dg})+2B_{\dg}(\Psi_{\dg}^{j-1},\Psi_{\dg},  \Phi_{\dg})-B_{\dg}(\Psi_{\dg}^{j-1},\Psi_{\dg}^{j-1},  \Phi_{\dg})\notag\\
	&=B_{\dg}(\Psi_{\dg}-\Psi_{\dg}^{j-1},\Psi_{\dg}^{j-1}-\Psi_{\dg},  \Phi_{\dg})\leq C_{\tilde{b}}\trinlNH \Psi_{\dg}-\Psi_{\dg}^{j-1}\trinrNH_{\dg}^2.\notag
	\end{align}
	This implies 
	\begin{equation}\label{induction0}
	\trinlNH\Psi_{\dg}-\Psi_{\dg}^{j}\trinrNH_{\dg}\leq \left(4C_{\tilde{b}}/\widehat{\beta}\right)\trinlNH\Psi_{\dg}-\Psi_{\dg}^{j-1}\trinrNH_{\dg}^2
	\end{equation}
	and establishes the quadratic convergence of the Newton method  to $\Psi_{\dg}$. The definition of $R$ and \eqref{induction0} guarantee $\trinlNH\Psi_{\dg}-\Psi_{\dg}^{j}\trinrNH_{\dg}\leq \half \trinlNH\Psi_{\dg}-\Psi_{\dg}^{j-1}\trinrNH_{\dg}<R$ to allow an induction step $j\to j+1$ to conclude the proof.
\end{proof}

\section[A posteriori error control]{{\it A posteriori} error control}

This section establishes a reliable and efficient error estimator for the DGFEM. For $K \in \cT$ and $E \in \cE(\Omega)$, define the volume and edge estimators $\eta_K$ and $\eta_E$ by 
\begin{align*}
\eta_K^2&:= h_K^4\Big{(}\|f+[u_{\dg},v_{\dg}]\|_{L^2(K)}^2+\|[u_{\dg},u_{\dg}]\|_{L^2(K)}^2\Big{)},\\
\eta_E^2&:={h_E\left(\left\|[D^2 u_{\dg}\,\nu_E]_E\right\|_{L^2(E)}^2+\left\|[ D^2 v_{\dg}\nu_E]_E\right\|_{L^2(E)}^2\right)}\\
&\quad {+h_E^{-3}\left(\|[u_{\dg}]_E\|_{L^2(E)}^2+\|[v_{\dg}]_E\|_{L^2(E)}^2\right)+h_E^{-1}\left(\|[\nabla u_{\dg}]_E\|_{L^2(E)}^2+\|[\nabla v_{\dg}]_E\|_{L^2(E)}^2\right).}
\end{align*}

\begin{thm}[Reliability]\label{reliability_dG} Let $\Psi=(u,v)\in\bX$ be a regular solution to \eqref{vform_cts_dG} and let $\Psi_{\dg}=(u_{\dg},v_{\dg})\in\Stb$ solve \eqref{vform_d_dG}. For sufficiently large $\sigma_2$ and sufficiently small $h$, there exists an $h$-independent positive constant $C_{\rm rel}$ such that
	\begin{align}
	\trinl\Psi-\Psi_{\dg}\trinr_{\dg}^2&\leq C_{\rm rel}^2\left(\sum_{K\in\cT}\eta_K^2+\sum_{E\in\cE(\Omega)}\eta_E^2\right). \label{reliability_Est_dG}
	\end{align}
\end{thm}

\begin{proof}
	The Fr\'{e}chet derivative of $ N(\Psi)$ at $\Psi$ in the direction $\Theta\in \bX$ reads
	\begin{align*}
	DN(\Psi;\Theta,\Phi):=A(\Theta,\Phi)+2B(\Psi,\Theta,\Phi)\fl\Phi\in \bX.
	\end{align*}
	Since $\Psi$ is a regular solution, for any $0<\epsilon<\beta$ with $\beta$ from \eqref{inf-sup_dG}, there exists some $\Phi\in \bX$ with $\trinl\Phi\trinr_2=1$ and
	\begin{align}\label{inf_sup_Epsilon}
	(\beta-\epsilon)\trinl \Psi-E_h\Psi_{\dg}\trinr_2\leq DN(\Psi;\Psi-E_h\Psi_{\dg},\Phi).
	\end{align}
	Since $N$ is quadratic, the finite Taylor series is exact and shows
	\begin{align}
	N(E_h\Psi_{\dg};\Phi)&=N(\Psi;\Phi)+DN(\Psi;E_h\Psi_{\dg}-\Psi,\Phi)\notag\\
	&\quad+\half D^2N(\Psi;E_h\Psi_{\dg}-\Psi, E_h\Psi_{\dg}-\Psi,\Phi). \label{Taylor_Exact}
	\end{align}
	Since $N(\Psi;\Phi)=0$ and $D^2N(\Psi;\Theta,\Theta,\Phi)=2B(\Theta,\Theta,\Phi)$ for $\Theta=\Psi-E_h\Psi_{\dg}$, \eqref{inf_sup_Epsilon}-\eqref{Taylor_Exact} plus Lemma~\ref{b_dG_1}.a with boundedness constant $C_b$ lead to
	\begin{align}
	(\beta-\epsilon)\trinl \Psi-E_h\Psi_{\dg}\trinr_2&\leq -N(E_h\Psi_{\dg};\Phi)+B(\Psi-E_h\Psi_{\dg}, \Psi-E_h\Psi_{\dg},\Phi)\notag\\
	&\leq |N(E_h\Psi_{\dg};\Phi)|+C_b\trinl \Psi-E_h\Psi_{\dg}\trinr_2^2.\label{app_infsup}
	\end{align}
	The triangle inequality, \eqref{EhLambda}, and Theorem \ref{eetimate} imply 
	\begin{equation}
	\trinl \Psi-E_h\Psi_{\dg}\trinr_{\dg}\leq \trinl\Psi-\Psi_{\dg}\trinr_{\dg}+\trinl \Psi_{\dg}-E_h\Psi_{\dg}\trinr_{\dg}
	\leq C(1+\Lambda)h^\alpha.\label{Cts2Dis}
	\end{equation}
	With $\epsilon\searrow 0$, \eqref{app_infsup}-\eqref{Cts2Dis} verify
	\begin{align*}
	\left(\beta-(1+\Lambda)CC_bh^\alpha\right)\trinl \Psi-E_h\Psi_{\dg}\trinr_2\leq |N(E_h\Psi_{\dg};\Phi)|.
	\end{align*}
	There exists positive constant $h_4$, such that $h\leq h_4$ implies $\beta-(1+\Lambda)CC_bh^\alpha>0$. Hence, for $h\leq h_4$, the above equation and triangle inequality lead to
	\begin{align*}
	\trinl\Psi-\Psi_{\dg}\trinr_{\dg}& \leq \trinl\Psi-E_h\Psi_{\dg}\trinr_{\dg}+\trinl E_h\Psi_{\dg}-\Psi_{\dg}\trinr_{\dg} \\
	&\lesssim |N(E_h\Psi_{\dg};\Phi)|+\trinl E_h\Psi_{\dg}-\Psi_{\dg}\trinr_{\dg}.
	\end{align*}
	The definitions of $N$ and  $N_h$, $N_h(\Psi_{\dg};\Pi_h\Phi)=0$, the boundedness of $A_{\dg}(\bullet, \bullet)$, $B_{\dg}(\bullet, \bullet, \bullet)$, and Theorem \ref{eetimate} result in
	\begin{align*}
	& N(E_h\Psi_{\dg};\Phi)=A(E_h\Psi_{\dg},\Phi)+B(E_h\Psi_{\dg},E_h\Psi_{\dg},\Phi)-L(\Phi)\notag\\
	&=N_h(\Psi_{\dg};\Phi)+A_{\dg}(E_h\Psi_{\dg}-\Psi_{\dg},\Phi)+B(E_h\Psi_{\dg},E_h\Psi_{\dg},\Phi)-B_{\dg}(\Psi_{\dg},\Psi_{\dg},\Phi)\notag\\
	&=N_h(\Psi_{\dg};\Phi-\Pi_h\Phi)+A_{\dg}(E_h\Psi_{\dg}-\Psi_{\dg},\Phi)\notag\\
	&\qquad+B_{\dg}(E_h\Psi_{\dg}-\Psi_{\dg},E_h\Psi_{\dg},\Phi)+B_{\dg}(\Psi_{\dg},E_h\Psi_{\dg}-\Psi_{\dg},\Phi)\notag\\
	&\lesssim |N_h(\Psi_{\dg};\Phi-\Pi_h\Phi)|+\trinl E_h\Psi_{\dg}-\Psi_{\dg}\trinr_{\dg}.
	\end{align*}
	\noindent The combination of the previous displayed estimates proves 
	\begin{equation}
	\trinl\Psi-\Psi_{\dg}\trinr_{\dg}\lesssim |N_h(\Psi_{\dg};\Phi-\Pi_h\Phi)|+\trinl E_h\Psi_{\dg}-\Psi_{\dg}\trinr_{\dg}. \label{final_apost_Est}
	\end{equation}	
	Abbreviate $\Phi-\Pi_h\Phi=:\boldsymbol{\chi}=(\chi_1,\chi_2)$ and recall
	\begin{align*}
	N_h(\Psi_{\dg};\boldsymbol{\chi})=A_{\dg}(\Psi_{\dg},\boldsymbol{\chi})+B_{\dg}(\Psi_{\dg},\Psi_{\dg},\boldsymbol{\chi})-L_{\dg}(\boldsymbol{\chi}).
	\end{align*}
	\noindent 
	The definition of $A_{\dg}(\bullet,\bullet)$ and an integration by parts with the facts that  $u_{\dg}$ and $v_{\dg}$ are piecewise quadratic polynomials lead to
	\begin{align*}
	&A_{\dg}(\Psi_{\dg},\boldsymbol{\chi})=\sik \left(D^2 u_{\dg}: D^2\chi_1+D^2 v_{\dg}: D^2\chi_2\right)\dx-J(u_{\dg},\chi_1)- J(\chi_1,u_{\dg})\notag\\
	&\qquad\qquad-J(v_{\dg},\chi_2)- J(\chi_2,v_{\dg}) +\sum_{E\in\cE}\frac{\sigma_1}{h_E^3}\int_E\Big{(} [u_{\dg}]_E[\chi_1]_E+[v_{\dg}]_E[\chi_2]_E\Big{)}\ds\notag\\
	&\qquad\qquad\quad+\sum_{E\in\cE}\frac{\sigma_2}{h_E}\int_E\Big{(}[\nabla u_{\dg}\cdot\nu_E]_E[\nabla \chi_1\cdot\nu_E]_E+[\nabla v_{\dg}\cdot\nu_E]_E[\nabla \chi_2\cdot\nu_E]_E\Big{)}\ds\notag\\
	&=\sie\Big{(}[D^2 u_{\dg}\,\nu_E]_E\cdot\langle \nabla\chi_1\rangle_E  +[D^2 v_{\dg}\,\nu_E ]_E\cdot\langle \nabla\chi_2\rangle_E  \Big{)}\ds\notag\\
	&\qquad-\left(J(\chi_1,u_{\dg})+J(\chi_2,v_{\dg})\right)
	+\sum_{E\in\cE}\frac{\sigma_1}{h_E^3}\int_E\Big{(}[u_{\dg}]_E[\chi_1]_E+[v_{\dg}]_E[\chi_2]_E\Big{)}\ds\notag\\
	&\qquad\quad+\sum_{E\in\cE}\frac{\sigma_2}{h_E}\int_E\Big{(}[\nabla u_{\dg}\cdot\nu_E]_E[\nabla \chi_1\cdot\nu_E]_E+[\nabla v_{\dg}\cdot\nu_E]_E[\nabla \chi_2\cdot\nu_E]_E\Big{)}\ds.
	\end{align*}
	\noindent The trace inequality of Lemma~\ref{trace_inq} and interpolation estimates $\|\chi_1\|_{H^m(K)}=\|\varphi_1-\Pi_h\varphi_1\|_{H^m(K)}\leq h_K^{2-m}\|\varphi_1\|_{H^2(K)}$ for $m=1,2$, from Lemma~\ref{interpolant_dG} result in
	\begin{align*}
	&\|h_{\cT}^{-1/2}\langle \nabla\chi_1\rangle_{\cE}  \|_{L^2(\Gamma)}^2\lesssim\sum_{K\in\cT}h_K^{-1}\|\nabla\chi_1\|_{L^2(\partial K)}^2\notag\\
	&\lesssim\sum_{K\in\cT}h_K^{-1}\left(h_K^{-1}\|\chi_1\|_{H^1(K)}^2+h_K\|\chi_1\|_{H^2(K)}^2\right)\lesssim\trinl\Phi\trinr_{2}^2=1. 
	\end{align*}
	This and the Cauchy-Schwarz inequality imply 
	\begin{align*}
	&\sie\Big{(}[D^2 u_{\dg}\,\nu_E]_E\cdot\langle \nabla\chi_1\rangle_E  +[D^2 v_{\dg}\,\nu_E ]_E\cdot\langle \nabla\chi_2\rangle_E  \Big{)}\ds\notag\\
	&\quad\lesssim \bigg{(}\sum_{E\in\cE(\Omega)}\|h_E^{1/2}[D^2 u_{\dg}\,\nu_E]_E\|_{L^2(E)}^2\bigg{)}^{1/2}+\bigg{(}\sum_{E\in\cE(\Omega)}\|h_E^{1/2}[D^2 v_{\dg}\,\nu_E]_E\|_{L^2(E)}^2\bigg{)}^{1/2}. 
	\end{align*}
	The Cauchy-Schwarz inequality, the trace inequality of Lemma~\ref{trace_inq}, and the interpolation of Lemma~\ref{interpolant_dG} lead to
	\begin{align*}
	&J(\chi_1,u_{\dg})+J(\chi_2,v_{\dg})\lesssim\left(\|h_\cE^{-1/2}[\nabla u_{\dg}]_E\|_{\Gamma}^2+\|h_\cE^{-1/2}[\nabla v_{\dg}]_E\|_{\Gamma}^2\right)^{1/2}\trinl {\Phi}\trinr_{2}. 
	\end{align*}
	Similar arguments for the penalty terms result in
	\begin{align*}
	&\sum_{E\in\cE}\frac{\sigma_1}{h_E^3}\int_E\Big{(} [u_{\dg}]_E[\chi_1]_E+[v_{\dg}]_E[\chi_2]_E\Big{)}\ds\lesssim \left(\|h_\cE^{-3/2}[u_{\dg}]_{\cE}\|_{L^2(\Gamma)}^2+\|h_\cE ^{-3/2}[v_{\dg}]_{\cE}\|_{L^2(\Gamma)}^2\right)^{1/2},\\
	&\sum_{E\in\cE}\frac{\sigma_2}{h_E}\int_E\Big{(}[\nabla u_{\dg}\cdot\nu_E]_E[\nabla \chi_1\cdot\nu_E]_E+[\nabla v_{\dg}\cdot\nu_E]_E[\nabla \chi_2\cdot\nu_E]_E\Big{)}\ds\notag\\
	&\qquad\qquad\qquad\qquad\lesssim \left(\|h_\cE^{-1}[\nabla u_{\dg}\cdot\nu_E]_{\cE}\|_{L^2(\Gamma)}^2+\|h_\cE ^{-1}[\nabla v_{\dg}\cdot\nu_E]_{\cE}\|_{L^2(\Gamma)}^2\right)^{1/2}.
	\end{align*}
	\noindent The definition of $B_{\dg}(\bullet,\bullet,\bullet)$ yields
	\begin{align*}
	B_{\dg}(\Psi_{\dg},\Psi_{\dg},\boldsymbol{\chi})
	&= -\sik[u_{\dg},v_{\dg}]\chi_1\dx+\half\sik[u_{\dg},u_{\dg}]\chi_2\dx.
	\end{align*}
    The Cauchy-Schwarz inequality and Lemma~\ref{interpolant_dG} show
	\begin{align*}
	&B_{\dg}(\Psi_{\dg},\Psi_{\dg},\boldsymbol{\chi})-L_{\dg}(\boldsymbol{\chi})\lesssim \sum_{K\in\cT} h_K^4\Big{(}\|f+[u_{\dg},v_{\dg}]\|_{L^2(K)}^2+\|[u_{\dg},u_{\dg}]\|_{L^2(K)}^2\Big{)}. 
	\end{align*}

	\noindent A substitution of the previous six estimates in \eqref{final_apost_Est} and Lemma~\ref{dG_Enrich_apost} establish \eqref{reliability_Est_dG}.
\end{proof}
\noindent The efficiency of the error estimator obtained in Theorem~\ref{reliability_dG} can be established using the standard bubble function techniques \cite{Verfurth,Georgoulis2011}; further details are therefore omitted.
\begin{thm}[Efficiency]\label{efficiency_dG}
	Let $\Psi=(u,v)\in\bX$ be a regular solution to \eqref{vform_cts_dG} and let $\Psi_{\dg}=(u_{\dg},v_{\dg})\in\Stb$ be the local solution to \eqref{vform_d_dG}. There exists a positive constant $C_{\rm eff}$ independent of $h$ but dependent on $\Psi$ such that
	\begin{align}\label{efficiency_Est_dG}
	\sum_{K\in\cT}\eta_K^2+\sum_{E\in\cE(\Omega)}\eta_E^2\leq C_{\rm eff}^2\Big{(}\trinl\Psi-\Psi_{\dg}\trinr_{\dg}^2+{\rm osc}^2(f)\Big{)},
	\end{align}
where $\displaystyle{\rm osc}^2(f):=\sum_{K\in\cT }h_K^4\|f-f_h\|_{L^2(K)}^2$ and $f_h$ denotes the piecewise average of $f$.
\end{thm}

\medskip
\section{A $C^0$ interior penalty method}\label{sect:C0IP}
The analysis of this paper extends to a $C^0$ interior penalty method for the von K\'{a}rm\'{a}n equations formally for $\sigma_1\to \infty$ when $\sigma_1$ disappears but the trial and test functions become continuous. The novel scheme is the above dG method but with ansatz test function restricted to $\Stb\cap \bH^1_0(\Omega)=:\bS^2_0(\cT)\equiv S^2_0(\cT)\times S^2_0(\cT)$ and the norm $\trinl\bullet\trinr_{\ip}$ is $\trinl\bullet\trinr_{\dg}$ with restriction to $\bS^2_0(\cT) $ excludes $\sigma_1$ (which has no meaning as it is multiplied by zero) and $\trinl\bullet\trinr_{\widetilde{\ip}}$ is $\trinl\bullet\trinr_{h}$ with restriction to $\bS^2_0(\cT)$.

\bigskip
Since the discrete functions are globally continuous for this case, the bilinear form  $a_{\dg}(\bullet,\bullet)$ simplifies for some positive penalty parameter $\sigma_2$, for $\eta_{\ip},\chi_{\ip}\in S^2_0(\cT)$,  to
\begin{align}
&a_{\ip}(\eta_{\ip},\chi_{\ip}):=\sik D^2\eta_{\ip}:D^2\chi_{\ip}\dx-\se \langle D^2\eta_{\ip}\nu_E\rangle_E  \cdot[\nabla\chi_{\ip}]_E\ds\notag\\
&\qquad-\se \langle D^2\chi_{\ip}\nu_E\rangle_E  \cdot[\nabla\eta_{\ip}]_E\ds+\sum_{E\in\cE}\frac{\sigma_2}{h_E}\int_E[\nabla\eta_{\ip}\cdot\nu_E]_E[\nabla\chi_{\ip}\cdot\nu_E]_E\ds.
\end{align}

This novel $C^0$ interior penalty ($C^0$-IP) method for the von K\'{a}rm\'{a}n equations seeks $u_{\ip},v_{\ip}\in S^2_0(\cT)$ such that
\begin{align}
&a_{\ip}(u_{\ip},\varphi_1)+b_{\dg}(u_{\ip},v_{\ip},\varphi_1)+b_{\dg}(v_{\ip},u_{\ip},\varphi_1)=l_{\dg}(\varphi_1)\fl\varphi_1\in S^2_0(\cT), \label{wformdg1_IP}\\
& a_{\ip}(v_{\ip},\varphi_2)-b_{\dg}(u_{\ip},u_{\ip},\varphi_2)=0\fl\varphi_2\in S^2_0(\cT). \label{wformdg2_IP}
\end{align}
The term related to jump which is of the form $[\eta_{\ip}]_E$ for each $\eta_{\ip}\in S^2_0(\cT)$ vanishes in the $C^0$-IP method and this simplifies the analysis. 
\begin{thm}[Energy norm estimate]
	Let  $\Psi$ be a regular solution to \eqref{vform_cts_dG}  and let $\Psi_{\ip}=(u_{\ip},v_{\ip})$ be the solution to \eqref{wformdg1_IP}-\eqref{wformdg2_IP}. For sufficiently large $\sigma_2$ and sufficiently small $h$, it holds
	\begin{equation*}
	\trinl\Psi-\Psi_{\ip}\trinr_{\ip}\leq C h^{\alpha}.
	\end{equation*}
\end{thm}

\smallskip
\begin{proof}
The Lemmas~\ref{interpolant_dG}-\ref{intermediate} hold as it is and the boundedness results in Lemma~\ref{b_dG_1} for $b_{\ip}(\bullet,\bullet,\bullet)$ can be modified to
\begin{align*}
b_{\ip}(\eta,\chi,\varphi)\lesssim
\begin{cases}
\|\eta\|_{\ip}\|\chi\|_{\ip}\|\varphi\|_{\ip} \fl \eta,\chi,\varphi\in X+S^2_0(\cT),\\
\|\eta\|_{2+\alpha}\|\chi\|_{\ip}\|\varphi\|_{1}\fl\eta\in X\cap H^{2+\alpha}(\Omega)\text{ and }\chi,\varphi\in X+S^2_0(\cT).
\end{cases}
\end{align*}
\smallskip
Theorems~\ref{dis_inf_sup_dG}-\ref{NewtonThm} follow in the same lines and hence, $a~priori$ error estimates in energy norm can be established without any additional difficulty.	
\end{proof}
For $K \in \cT$ and $E \in \cE(\Omega)$, $a~posteriori$ error estimates for the $C^0$-interior penalty method \eqref{wformdg1_IP}-\eqref{wformdg2_IP} lead to
the volume estimator $\eta_K$ and the edge estimator $\eta_E$ defined by
\begin{align*}
\eta_K^2&:= h_K^4\Big{(}\|f+[u_{\ip},v_{\ip}]\|_{L^2(K)}^2+\|[u_{\ip},u_{\ip}]\|_{L^2(K)}^2\Big{)},\\
\eta_E^2&:={h_E\left(\|[D^2 u_{\ip}\,\nu_E]_E\|_{L^2(E)}^2+\|[ D^2 v_{\ip}\nu_E]_E\|_{L^2(E)}^2\right)}\\
&\qquad{+h_E^{-1}\left(\|[\nabla u_{\ip}]\|_{L^2(E)}^2+\|[\nabla v_{\ip}]\|_{L^2(E)}^2\right).}
\end{align*}
\begin{thm}\label{reliability_IP} Let $\Psi=(u,v)\in\bX$ be a regular solution to \eqref{vform_cts_dG} and $\Psi_{\ip}=(u_{\ip},v_{\ip})\in S^2_0(\cT)\times S^2_0(\cT)$ be the solution to \eqref{wformdg1_IP}-\eqref{wformdg2_IP}. For sufficiently large $\sigma_2$ and sufficiently small $h$, there exist $h$-independent positive constants $C_{\rm rel}$ and $C_{\rm eff}$ such that
	\begin{align}\label{rel_eff_IP}
	C_{\rm rel}^{-2}\trinl\Psi-\Psi_{\ip}\trinr_{\ip}^2&\leq \sum_{K\in\cT}\eta_K^2+\sum_{E\in\cE(\Omega)}\eta_E^2\leq C_{\rm eff}^2\trinl\Psi-\Psi_{\ip}\trinr_{\ip}^2+{\rm osc}^2(f).
	\end{align}
\end{thm}
\begin{proof}
	\noindent This follows as in  Theorems~\ref{reliability_dG}-\ref{efficiency_dG}, and hence the details are omitted for brevity.
\end{proof}
\begin{rem}\label{Brenner_b}
	\noindent The $C^0$-IP formulation of \cite{BS_C0IP_VKE} chooses the trilinear form $b_{\widetilde{\ip}}(\bullet,\bullet,\bullet)$ as
	\begin{align} 
	b_{\widetilde{\ip}}(\eta_{\ip},\chi_{\ip},\varphi_{\ip})&:=-\half\sik [\eta_{\ip},\chi_{\ip}]\varphi_{\ip}\dx\notag\\
	&\qquad+\half\sum_{E\in \cE(\Omega)}\int_E
	\left[\langle \cof(D^2\eta_{\ip})\rangle_E  \nabla\chi_{\ip}\cdot\nu_E\right]_E\varphi_{\ip}\ds
	\label{BS_IP_Form}
	\end{align}
		\noindent for all $\eta_{\ip},\chi_{\ip},\varphi_{\ip}\in S^2_0(\cT)$.
	For the $C^0$-IP formulation \eqref{wformdg1_IP}-\eqref{wformdg2_IP} with $	b_{\widetilde{\ip}}(\bullet,\bullet,\bullet)$ replacing $	b_{\ip}(\bullet,\bullet,\bullet)$ and $\|\bullet\|_{h}\equiv\|\bullet\|_{\widetilde{\ip}}$, the efficiency of the estimator is still open, for difficulties caused by the the non-residual type average term $\langle \cof(D^2\eta_{\ip})\rangle_E$.
\end{rem}

\section{Numerical experiments}
This section is devoted to numerical experiments to investigate the practical parameter choice and adaptive mesh-refinements.

\subsection{Preliminaries}
The discrete solution to \eqref{vform_d_dG} is obtained using the Newton method defined in \eqref{NewtonIterate_dG} with initial guess $\Psi_{\dg}^0\in\Stb$ computed as the solution of the biharmonic part of the von K\'{a}rm\'{a}n equations, i.e., $\Psi_{\dg}^0\in\Stb$ solves
\begin{align}\label{BihInitial}
A_{\dg}(\Psi_{\dg}^0,\Phi_{\dg})=L(\Phi_{\dg})\fl\Phi_{\dg}\in\Stb.
\end{align}
Let the $\ell$-th level error (for example, in the norm $\trinl\Psi-\Psi_{\dg}\trinr_{\dg}$) and the number of degrees of freedom (ndof) be denoted by $e_{\ell}$ and $\texttt{ndof}(\ell)$, respectively. The $\ell$-th level empirical rate of convergence is defined by
\begin{equation*}
    \texttt{rate}(\ell):=2\times \log \big(e_{\ell-1}/e_{\ell} \big)/\log \big(\texttt{ndof}(\ell)/\texttt{ndof}(\ell-1) \big)\text{ for } \ell=1,2,3,\ldots
\end{equation*}
In all the numerical tests, the Newton iterates converge within $4$ steps with the stopping criteria $\trinl\Psi_{\dg}^{5}-\Psi_{\dg}^{j-1}\trinr_{\dg}<10^{-8}$ for $j\in\bN$, where $\Psi_{\dg}^{5}$ denotes the discrete solution generated by Newton iterates at $5$-th iteration. The penalty parameters for the DGFEM and $C^0$-IP are consistently chosen 
as $\sigma_1=\sigma_2=20$ in all numerical examples and appear as sensitive as in the case of the linear biharmonic equations. 

\subsection{Example on unit square domain}\label{example_UnitSq}
The exact solution to \eqref{vkedG}  is $u(x,y)=x^2y^2(1-x)^2(1-y)^2$ and 
$v(x,y)=\sin^2(\pi x)\sin^2(\pi y) $ on the unit square $\Omega$
with elliptic regularity index $\alpha=1$ and corresponding data $f$ and $g$. 
Figure~\ref{fig:UnitSq_mesh} displays  the initial mesh and its successive 
red-refinements lead to sequence of  DGFEM solutions on the quasi-uniform meshes 
with the errors $\trinl  u-u_{\dg}\trinr_{\dg}$ and $\trinl v-v_{\dg}\trinr_{\dg}$ and with their
empirical convergence rates in Table~\ref{Table:UnitSq_DG}. The empirical convergence rate  with respect to dG norm is one as predicted from Theorem~\ref{eetimate}. 
Table~\ref{Table:UnitSq_C0IP} displays the errors  and empirical rates of convergence for $C^0$-IP method of Section~\ref{sect:C0IP} and the linear empirical rate of convergence as expected from the theory is observed.
\begin{figure}
	\begin{center}
		\subfloat[]{\includegraphics[height=2.1in,width=2.25in,angle=0]{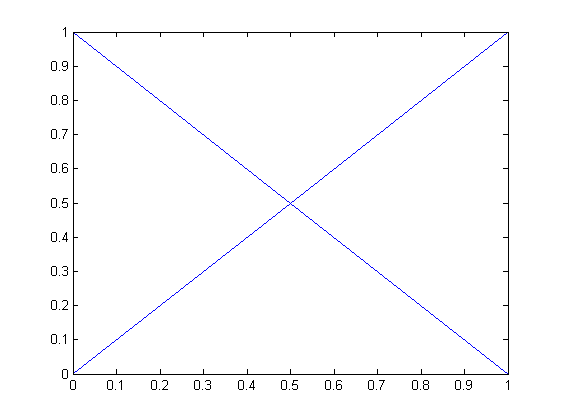}}
		\subfloat[]{\includegraphics[height=2.1in,width=2.25in,angle=0]{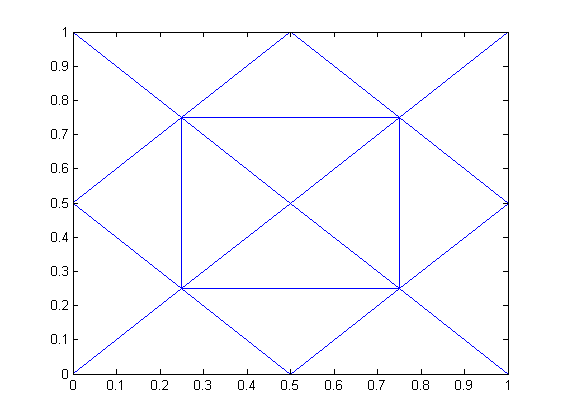}}
		\caption{(a) Initial triangulation and (b) refined triangulation for unit square domain}
		\label{fig:UnitSq_mesh}
	\end{center}
\end{figure} 

\begin{table}[h!]
	\begin{center}
		\begin{tabular}{|c| r r |r r|}
			\hline
\texttt{ndof}&$\|u-u_{\dg}\|_{\dg}$&\texttt{rate}&$\|v-v_{\dg}\|_{\dg}$&\texttt{rate}\\
			\hline	
80   &0.0577836157 &- &12.7526558280 &-\\
352  &0.0310832754 &0.8369 & 7.4399026157 &0.7274\\
1472 &0.0137172531 &1.1434 & 3.4111200164 &1.0900\\	
6016 &0.0061972815 &1.1287 & 1.4811727729 &1.1851\\	
24320&0.0029481076 &1.0637 & 0.6762547625 &1.1225\\		
\hline			
		\end{tabular}
	\end{center}
\caption{Errors and  empirical rates of convergence for DGFEM of Example~\ref{example_UnitSq}}
\label{Table:UnitSq_DG}
\end{table}		
\begin{table}[h!]
	\begin{center}
		\begin{tabular}{|c| r r |r r|}
			\hline
\texttt{ndof}&$\|u-u_{\ip}\|_{\dg}$&\texttt{rate}&$\|v-v_{\ip}\|_{\ip}$&\texttt{rate}\\
			\hline	

25   &0.0461238157 &- &11.2692445826 &-\\
113  &0.0218618458 &0.9898 & 4.9052656166 &0.7274\\
481  &0.0103962085 &1.0263 & 2.4140866870 &0.9789\\	
1985 &0.0049162143 &1.0566 & 1.1624980996 &1.0310\\	
8065 &0.0023958471 &1.0254 & 0.5703487012 &1.0158\\		
32513&0.0011879209 &1.0064 & 0.2834173301 &1.0032\\
\hline			
	 \end{tabular}
	\end{center}
\caption{Errors and  empirical rates of convergence for $C^0$-IP method of Example~\ref{example_UnitSq}}
\label{Table:UnitSq_C0IP}
\end{table}

\subsection{Example on L-shaped domain}\label{singular_xmpl_dg} 
In polar coordinates centered at the re-entering corner of the  
 L-shaped domain $\Omega=(-1,1)^2 \setminus\big{(}[0,1)\times(-1,0]\big{)}$,
 the slightly singular functions 
$\displaystyle
	u(r,\theta)=v(r,\theta):=(1-r^2 \cos^2\theta)^2 (1-r^2 \sin^2\theta)^2 r^{1+\alpha}g_{\alpha,\omega}(\theta)$ 	with the abbreviation 
$	g_{\alpha,\omega}(\theta):=$
	\begin{align*}
	&\left(\frac{1}{\alpha-1}\sin\big{(}(\alpha-1)\omega\big{)}-\frac{1}{\alpha+1}\sin\big{(}(\alpha+1)\omega\big{)}\right)\times\Big{(}\cos\big{(}(\alpha-1)\theta\big{)}-\cos\big{(}(\alpha+1)\theta\big{)}\Big{)}\\
	&-\left(\frac{1}{\alpha-1}\sin\big{(}(\alpha-1)\theta\big{)}-\frac{1}{\alpha+1}\sin\big{(}(\alpha+1)\theta\big{)}\right)\times\Big{(}\cos\big{(}(\alpha-1)\omega\big{)}-\cos\big{(}(\alpha+1)\omega\big{)}\Big{)},
	\end{align*}
are defined for the angle $\omega=\frac{3\pi}{2}$ and the parameter
$\alpha= 0.5444837367$ as the non-characteristic 
	root of $\sin^2(\alpha\omega) = \alpha^2\sin^2(\omega)$.
With the loads $f$ and $g$ according to \eqref{vkedG}
the DGFEM solutions are computed on a sequence of quasi-uniform meshes. 
Tables~\ref{Table:Lshaped_DG} and ~\ref{Table:Lshaped_C0IP}  
display the errors and the expected reduced
empirical convergence rates for the DGFE and the $C^0$-IP method. 
\begin{table}[h!]
	\begin{center}
		\begin{tabular}{|c| r r |r r|}
			\hline
\texttt{ndof}&$\|u-u_{\dg}\|_{\dg}$&\texttt{rate}&$\|v-v_{\dg}\|_{\dg}$&\texttt{rate}\\
			\hline		
112  &15.2305842271 &- &15.2984276544 &-\\
512  & 8.0774597721 &0.8346 & 7.8781411333 &0.8733\\
2176 & 4.1100793036 &0.9338 & 4.1061231414 &0.9006\\	
8960 & 2.1174046583 &0.9372 & 2.1241422224 &0.9314\\	
36352& 1.1536162513 &0.8761 & 1.1556461785 &0.8781\\		
\hline			
		\end{tabular}
	\end{center}
\caption{Errors and  empirical rate of convergence for DGFEM of Example~\ref{singular_xmpl_dg}}
\label{Table:Lshaped_DG}
\end{table}
\begin{table}[h!]
	\begin{center}
		\begin{tabular}{|c| r r |r r|}
			\hline
\texttt{ndof}&$\|u-u_{\ip}\|_{\ip}$&\texttt{rate}&$\|v-v_{\ip}\|_{\ip}$&\texttt{rate}\\
			\hline	
33   &10.3829891715 &-&10.317434853 &-\\
161  & 6.7186797840 &0.5492 & 6.6122910205 &0.5614\\
705  & 3.5483398032 &0.8645 & 3.5119042793 &0.8569\\	
2945 & 1.8327303743 &0.9242 & 1.8246293348 &0.9159\\	
12033& 0.9771452705 &0.8936 & 0.9756119144 &0.8895\\
48641& 0.5448961726 &0.8362 & 0.5446327773 &0.8346\\		
\hline			
		\end{tabular}
	\end{center}
\caption{Errors and  empirical rate of convergence for $C^0$-IP method of Example~\ref{singular_xmpl_dg}}
\label{Table:Lshaped_C0IP}
\end{table}

\subsection{Adaptive mesh-refinement}\label{AdaptMesh_Esample}
For the  L-shaped domain of the preceding Example~\ref{singular_xmpl_dg}
the constant load function $f \equiv 1$, the unknown solution 
to the \vket \eqref{vkedG} is approximated by an adaptive mesh-refining algorithm. 

Given an initial  triangulation $\cT_0$ run the steps
{\bf SOLVE, ESTIMATE, MARK} and {\bf REFINE} successively 
for different levels $\ell=0,1,2,\ldots$.

\medskip
\noindent {\bf SOLVE} Compute the solution of DGFEM $\Psi_{\ell}:=\Psi_{\dg}$ (resp. $C^0$-IP $\Psi_{\ell}:=\Psi_{\ip}$ ) with respect to $\cT_{\ell}$ and number of degrees of freedom given by \texttt{ndof}.\\

\medskip
\noindent{\bf ESTIMATE} Compute local contribution of the error estimator from \eqref{reliability_Est_dG} (resp. from \eqref{rel_eff_IP}) 
\begin{equation*}
\eta^2_{\ell}(K):=\eta_K^2+\sum_{E\in \cE(K)}\eta_E^2 \quad\text{ for all } K\in\cT_{\ell}.
\end{equation*}

\medskip
\noindent{\bf MARK} The D\"{o}rfler marking chooses a minimal subset $\mathcal{M}_{\ell}\subset \cT_{\ell}$ such that 
$$0.3\, \sum_{K\in\mathcal{T}_{\ell}}\eta^2_{\ell}(K)\leq \sum_{K\in\mathcal{M}_{\ell}}\eta^2_{\ell}(K). $$

\medskip
\noindent{\bf REFINE} Compute the closure of $\mathcal{M}_{\ell}$ and generate a new triangulation $\cT_{\ell+1}$ using newest vertex bisection \cite{Stevenson08}.

\medskip

\noindent  Figure~\ref{fig:Lshape_f1_Conv}(a) displays the  convergence history of 
the {\it a posteriori } error estimator  as a function of number of degrees of freedom for uniform and adaptive mesh-refinement of DGFEM and $C^0$-IP method.
 
Figure~\ref{fig:Lshape_f1_Conv}(b) depicts the adaptive mesh for $C^0$-IP method generated by the above adaptive algorithm for level $\ell=19$, and it illustrates the adaptive mesh-refinement near the re-entering corner. The suboptimal empirical convergence rate for uniform mesh-refinement is improved to an optimal empirical convergence rate 0.5 via adaptive mesh-refinement.

\medskip
\noindent To show the reliability and efficiency of the estimators for DGFEM and $C^0$-IP, another test has been performed over L-shaped domain for the example described in Example~\ref{singular_xmpl_dg}. Figure~\ref{fig:DG_C0IP_Lshaped_adap0p3}(a) displays the convergence history of the error and the {\it a posteriori } error estimator as a function of number of degrees of freedom for uniform and adaptive mesh-refinement of DGFEM and $C^0$-IP method. 
Figure~\ref{fig:DG_C0IP_Lshaped_adap0p3}(b) displays the convergence history of 
the error and the {\it a posteriori } error estimator for uniform and 
adaptive mesh-refinement of $C^0$-IP method. 
The ratio between error and estimator $C_{\text{rel}}$  is plotted in 
Figure~\ref{fig:DG_C0IP_Lshaped_adap0p3}(a)-(b) and almost constant as a numerical evidence of the reliability and efficiency of the estimators for DGFEM 
and $C^0$-IP methods of Theorem~\ref{reliability_Est_dG}-\ref{efficiency_dG} and Theorem~\ref{reliability_IP}.

\begin{figure}
	\begin{center}
	\subfloat[]{\includegraphics[height=3in,width=3.5in,angle=0]{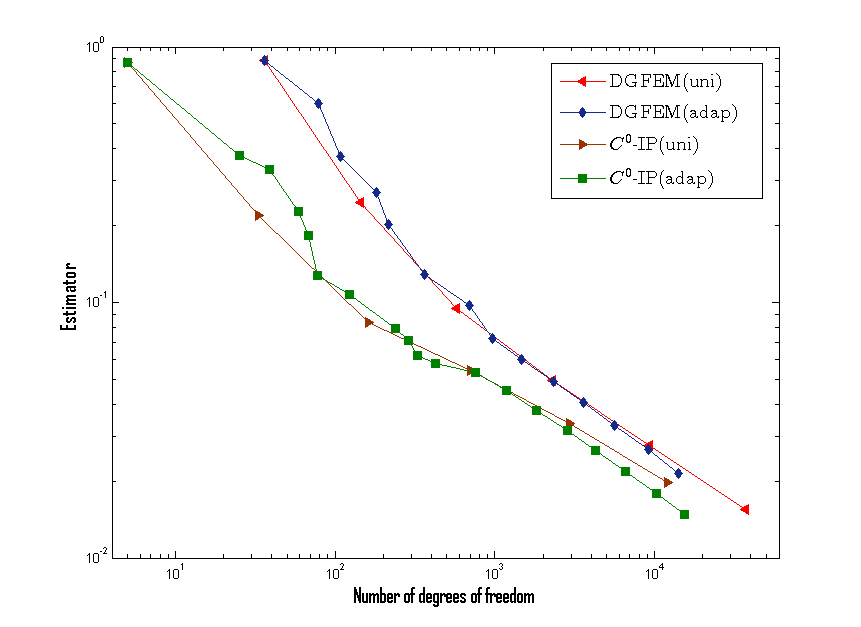}}
	\subfloat[]{\includegraphics[height=3in,width=3in,angle=0]{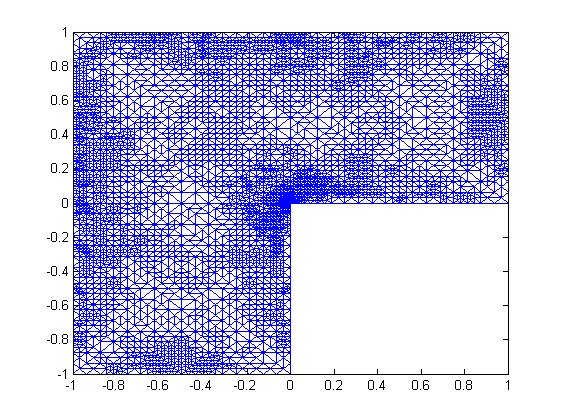}}
		\caption{(a) Convergence history for DGFEM and $C^0$-IP method of Example~\ref{AdaptMesh_Esample} with $f\equiv 1$ and (b) adaptive mesh for $C^0$-IP method at the level $\ell=19$.}
		\label{fig:Lshape_f1_Conv}
	\end{center}
\end{figure}

\begin{figure}
	\begin{center}
		\subfloat[]{\includegraphics[height=3in,width=3.2in,angle=0]{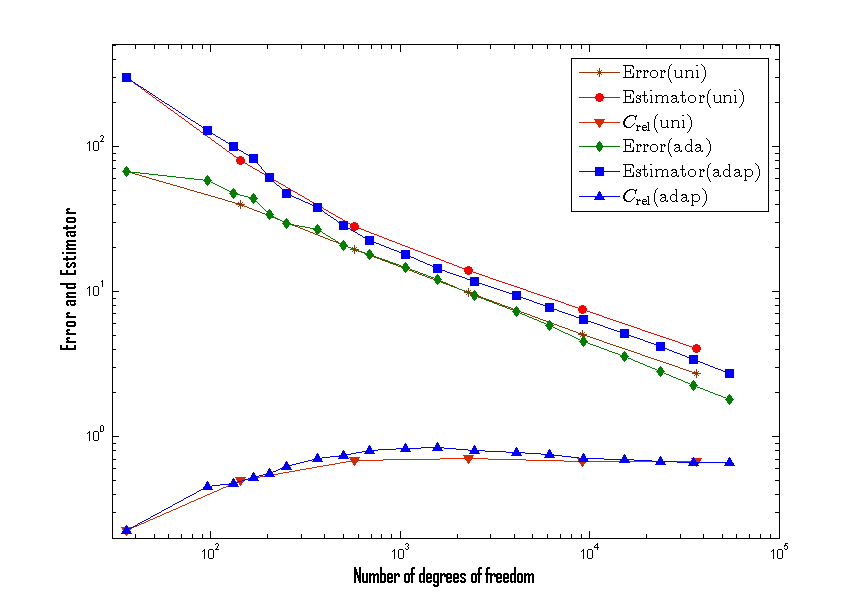}}
		\subfloat[]{\includegraphics[height=3in,width=3.2in,angle=0]{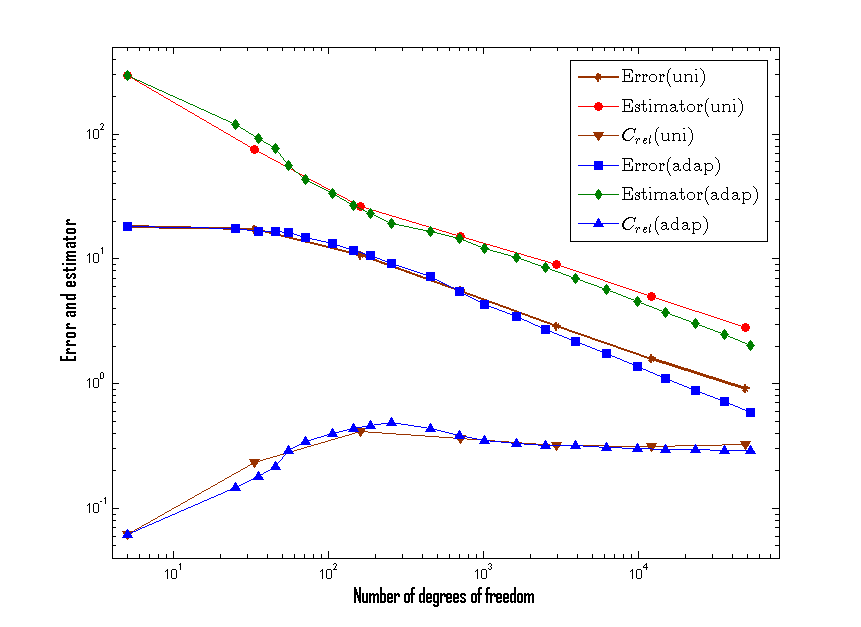}}
		\caption{Convergence history of {\it a posteriori} error control
		for (a) DGFEM and (b) $C^0$-IP method}
		\label{fig:DG_C0IP_Lshaped_adap0p3}
	\end{center}
\end{figure}

\section{Conclusions}
This paper analyzes a discontinuous Galerkin finite element method for the approximation of regular solutions of von K\'{a}rm\'{a}n equations. A priori error estimate in energy norm and a posteriori error control which motivates an adaptive mesh-refinement are deduced under the minimal regularity assumption on the exact solution. The analysis suggests a novel $C^0$-interior penalty method and provides {\it a priori} and {\it a posteriori} error control for the energy norm.  Moreover, the analysis can be extended to $hp$ discontinuous Galerkin finite element methods with additional jump terms for higher order derivatives of ansatz and trial functions under additional regularity assumptions on the exact solution.

%

\section*{Acknowledgements}
{The first and third authors acknowledge the support of National Program on Differential Equations: Theory, Computation \& Applications (NPDE-TCA) and Department of Science \& Technology (DST) Project No. SR/S4/MS:639/09. The second author would like to thank National Board for Higher Mathematics (NBHM) and IIT Bombay for various financial supports.
}
\bibliographystyle{amsplain}
\bibliography{refs}

\end{document}